\documentclass[a4paper,11pt,reqno]{amsart}

\usepackage[utf8x]{inputenc}
\usepackage{hyperref}
\usepackage{xcolor}
\hypersetup{
    colorlinks,
    linkcolor={red!50!black},
    citecolor={blue!50!black},
    urlcolor={blue!80!black}
}

\usepackage{microtype}
\usepackage{MnSymbol}
\usepackage{textcomp}

\makeatletter
\newcommand{\eqnum}{\refstepcounter{equation}\textup{\tagform@{\theequation}}}
\makeatother

\makeatletter \@addtoreset{equation}{section} \makeatother
\renewcommand{\theequation}{\thesection.\arabic{equation}}
\newtheorem{thm}[equation]{Theorem}
\newtheorem*{thm*}{Theorem}
\newtheorem{lem}[equation]{Lemma}
\newtheorem{cor}[equation]{Corollary}
\newtheorem{prop}[equation]{Proposition}
\newtheorem*{defthm*}{Definition/Theorem}

\theoremstyle{definition}
\newtheorem{defn}[equation]{Definition}
\newtheorem{rem}[equation]{Remark}
\newtheorem{exam}[equation]{Example}
\newtheorem{constr}[equation]{Construction}

\newtheorem{warn}[equation]{Warning}
\newtheorem*{exam*}{Example}

\newcommand\arXiv[1]{\href{http://arxiv.org/abs/#1}{arXiv:#1}}

\usepackage{xparse}

\usepackage{xspace}

\newcommand{\changelocaltocdepth}[1]{%
  \addtocontents{toc}{\protect\setcounter{tocdepth}{#1}}%
  \setcounter{tocdepth}{#1}}

\newcommand{\nc}{\newcommand}
\nc{\renc}{\renewcommand}
\nc{\ssec}{\subsection}
\nc{\sssec}{\subsubsection}
\nc{\on}{\operatorname}
\nc{\term}[1]{#1\xspace}

\usepackage{tikz}
\usetikzlibrary{matrix}
\usepackage{tikz-cd}

\tikzset{
  commutative diagrams/.cd,
  arrow style=tikz,
  diagrams={>=latex}}
\makeatletter
\tikzset{
  column sep/.code=\def\pgfmatrixcolumnsep{\pgf@matrix@xscale*(#1)},
  row sep/.code   =\def\pgfmatrixrowsep{\pgf@matrix@yscale*(#1)},
  matrix xscale/.code=%
    \pgfmathsetmacro\pgf@matrix@xscale{\pgf@matrix@xscale*(#1)},
  matrix yscale/.code=%
    \pgfmathsetmacro\pgf@matrix@yscale{\pgf@matrix@yscale*(#1)},
  matrix scale/.style={/tikz/matrix xscale={#1},/tikz/matrix yscale={#1}}}
\def\pgf@matrix@xscale{1}
\def\pgf@matrix@yscale{1}
\makeatother

\usepackage[inline]{enumitem}
\setlist[enumerate,1]{label={(\alph*)},itemsep=\parskip}

\newlist{thmlist}{enumerate}{1}
\setlist[thmlist,1]{
  label={\em(\roman*)}, ref={(\roman*)},
  itemsep=0.5em,
  align=right,widest=vi)}

\newlist{thmlistbis}{enumerate}{1}
\setlist[thmlistbis,1]{
  label={\em(\roman*~\textit{bis})},
  ref={(\roman*}~\textit{bis}\upshape{)},
  itemsep=0.5em,
  leftmargin=0pt, align=right, widest=vi)}

\newlist{defnlist}{enumerate}{2}
\setlist[defnlist,1]{
  label={(\roman*)}, ref={(\roman*)},
  itemsep=0.5em,
  align=right, widest=vi)}

\setlist[defnlist,2]{
  label={(\alph*)}, ref={(\alph*)},
  itemsep=0.75em,
  labelsep=0em,labelindent=0em,leftmargin=*,align=left,widest=vi),
  topsep=0.75em}

\newlist{inlinelist}{enumerate*}{1}
\setlist[inlinelist,1]{label={(\alph*)}}

\newlist{inlinedefnlist}{enumerate*}{1}
\definecolor{green}{HTML}{38550C}
\setlist[inlinedefnlist,1]{label={\color{green}(\roman*)}}

\nc{\sA}{\ensuremath{\mathcal{A}}\xspace}
\nc{\sB}{\ensuremath{\mathcal{B}}\xspace}
\nc{\sC}{\ensuremath{\mathcal{C}}\xspace}
\nc{\sD}{\ensuremath{\mathcal{D}}\xspace}
\nc{\sE}{\ensuremath{\mathcal{E}}\xspace}
\nc{\sF}{\ensuremath{\mathcal{F}}\xspace}
\nc{\sG}{\ensuremath{\mathcal{G}}\xspace}
\nc{\sH}{\ensuremath{\mathcal{H}}\xspace}
\nc{\sI}{\ensuremath{\mathcal{I}}\xspace}
\nc{\sJ}{\ensuremath{\mathcal{J}}\xspace}
\nc{\sK}{\ensuremath{\mathcal{K}}\xspace}
\nc{\sL}{\ensuremath{\mathcal{L}}\xspace}
\nc{\sM}{\ensuremath{\mathcal{M}}\xspace}
\nc{\sN}{\ensuremath{\mathcal{N}}\xspace}
\nc{\sO}{\ensuremath{\mathcal{O}}\xspace}
\nc{\sP}{\ensuremath{\mathcal{P}}\xspace}
\nc{\sQ}{\ensuremath{\mathcal{Q}}\xspace}
\nc{\sR}{\ensuremath{\mathcal{R}}\xspace}
\nc{\sS}{\ensuremath{\mathcal{S}}\xspace}
\nc{\sT}{\ensuremath{\mathcal{T}}\xspace}
\nc{\sU}{\ensuremath{\mathcal{U}}\xspace}
\nc{\sV}{\ensuremath{\mathcal{V}}\xspace}
\nc{\sW}{\ensuremath{\mathcal{W}}\xspace}
\nc{\sX}{\ensuremath{\mathcal{X}}\xspace}
\nc{\sY}{\ensuremath{\mathcal{Y}}\xspace}
\nc{\sZ}{\ensuremath{\mathcal{Z}}\xspace}

\nc{\bA}{\ensuremath{\mathbf{A}}\xspace}
\nc{\bB}{\ensuremath{\mathbf{B}}\xspace}
\nc{\bC}{\ensuremath{\mathbf{C}}\xspace}
\nc{\bD}{\ensuremath{\mathbf{D}}\xspace}
\nc{\bE}{\ensuremath{\mathbf{E}}\xspace}
\nc{\bF}{\ensuremath{\mathbf{F}}\xspace}
\nc{\bG}{\ensuremath{\mathbf{G}}\xspace}
\nc{\bH}{\ensuremath{\mathbf{H}}\xspace}
\nc{\bI}{\ensuremath{\mathbf{I}}\xspace}
\nc{\bJ}{\ensuremath{\mathbf{J}}\xspace}
\nc{\bK}{\ensuremath{\mathbf{K}}\xspace}
\nc{\bL}{\ensuremath{\mathbf{L}}\xspace}
\nc{\bM}{\ensuremath{\mathbf{M}}\xspace}
\nc{\bN}{\ensuremath{\mathbf{N}}\xspace}
\nc{\bO}{\ensuremath{\mathbf{O}}\xspace}
\nc{\bP}{\ensuremath{\mathbf{P}}\xspace}
\nc{\bQ}{\ensuremath{\mathbf{Q}}\xspace}
\nc{\bR}{\ensuremath{\mathbf{R}}\xspace}
\nc{\bS}{\ensuremath{\mathbf{S}}\xspace}
\nc{\bT}{\ensuremath{\mathbf{T}}\xspace}
\nc{\bU}{\ensuremath{\mathbf{U}}\xspace}
\nc{\bV}{\ensuremath{\mathbf{V}}\xspace}
\nc{\bW}{\ensuremath{\mathbf{W}}\xspace}
\nc{\bX}{\ensuremath{\mathbf{X}}\xspace}
\nc{\bY}{\ensuremath{\mathbf{Y}}\xspace}
\nc{\bZ}{\ensuremath{\mathbf{Z}}\xspace}

\nc{\bbA}{\ensuremath{\mathbb{A}}\xspace}
\nc{\bbB}{\ensuremath{\mathbb{B}}\xspace}
\nc{\bbC}{\ensuremath{\mathbb{C}}\xspace}
\nc{\bbD}{\ensuremath{\mathbb{D}}\xspace}
\nc{\bbE}{\ensuremath{\mathbb{E}}\xspace}
\nc{\bbF}{\ensuremath{\mathbb{F}}\xspace}
\nc{\bbG}{\ensuremath{\mathbb{G}}\xspace}
\nc{\bbH}{\ensuremath{\mathbb{H}}\xspace}
\nc{\bbI}{\ensuremath{\mathbb{I}}\xspace}
\nc{\bbJ}{\ensuremath{\mathbb{J}}\xspace}
\nc{\bbK}{\ensuremath{\mathbb{K}}\xspace}
\nc{\bbL}{\ensuremath{\mathbb{L}}\xspace}
\nc{\bbM}{\ensuremath{\mathbb{M}}\xspace}
\nc{\bbN}{\ensuremath{\mathbb{N}}\xspace}
\nc{\bbO}{\ensuremath{\mathbb{O}}\xspace}
\nc{\bbP}{\ensuremath{\mathbb{P}}\xspace}
\nc{\bbQ}{\ensuremath{\mathbb{Q}}\xspace}
\nc{\bbR}{\ensuremath{\mathbb{R}}\xspace}
\nc{\bbS}{\ensuremath{\mathbb{S}}\xspace}
\nc{\bbT}{\ensuremath{\mathbb{T}}\xspace}
\nc{\bbU}{\ensuremath{\mathbb{U}}\xspace}
\nc{\bbV}{\ensuremath{\mathbb{V}}\xspace}
\nc{\bbW}{\ensuremath{\mathbb{W}}\xspace}
\nc{\bbX}{\ensuremath{\mathbb{X}}\xspace}
\nc{\bbY}{\ensuremath{\mathbb{Y}}\xspace}
\nc{\bbZ}{\ensuremath{\mathbb{Z}}\xspace}

\nc{\mrm}[1]{\ensuremath{\mathrm{#1}}\xspace}
\nc{\mit}[1]{\ensuremath{\mathit{#1}}\xspace}
\nc{\mbf}[1]{\ensuremath{\mathbf{#1}}\xspace}
\nc{\mcal}[1]{\ensuremath{\mathcal{#1}}\xspace}
\nc{\msc}[1]{\ensuremath{\mathscr{#1}}\xspace}

\nc{\sub}{\subseteq}
\nc{\too}{\longrightarrow}
\nc{\hook}{\hookrightarrow}
\nc{\hooklongrightarrow}{\lhook\joinrel\longrightarrow}
\nc{\hooklong}{\hooklongrightarrow}
\nc{\hooklongleftarrow}{\longleftarrow\joinrel\rhook}
\nc{\twoheadlongrightarrow}{\relbar\joinrel\twoheadrightarrow}
\nc{\longrightleftarrows}{\ \raisebox{0.3ex}{\(\mathrel{\substack{\xrightarrow{\rule{1em}{0em}} \\[-1ex] \xleftarrow{\rule{1em}{0em}}}}\)}\ }

\renc{\ge}{\geqslant}
\renc{\le}{\leqslant}

\nc{\id}{\mathrm{id}}

\DeclareMathOperator{\Hom}{\on{Hom}}
\nc{\uHom}{\underline{\smash{\Hom}}}
\DeclareMathOperator{\Maps}{\on{Maps}}

\DeclareMathOperator{\End}{\on{End}}

\nc{\uEnd}{\underline{\smash{\End}}}

\nc{\colim}{\varinjlim}
\renc{\lim}{\varprojlim}
\nc{\Cofib}{\on{Cofib}}
\nc{\Fib}{\on{Fib}}
\nc{\initial}{\varnothing}
\nc{\op}{\mathrm{op}}

\DeclareMathOperator*{\fibprod}{\times}

\renc{\setminus}{\smallsetminus}

\usepackage{mathtools}
\DeclarePairedDelimiter\abs{\lvert}{\rvert}%

\newcommand{\thmref}[1]{Theorem~\ref{#1}}

\newcommand{\ssecref}[1]{Subsect. ~\ref{#1}}

\newcommand{\lemref}[1]{Lemma~\ref{#1}}
\newcommand{\propref}[1]{Proposition~\ref{#1}}
\newcommand{\corref}[1]{Corollary~\ref{#1}}

\newcommand{\remref}[1]{Remark~\ref{#1}}
\newcommand{\defnref}[1]{Definition~\ref{#1}}

\renewcommand{\eqref}[1]{(\ref{#1})}
\newcommand{\constrref}[1]{Construction~\ref{#1}}

\newcommand{\examref}[1]{Example~\ref{#1}}
\newcommand{\warnref}[1]{Warning~\ref{#1}}
\newcommand{\itemref}[1]{\ref{#1}}

\nc{\A}{\bA}
\renc{\P}{\bP}
\nc{\Spec}{\on{Spec}}
\nc{\D}{\on{\mbf{D}}}
\nc{\Dqc}{\on{\mbf{D}}_{\mrm{qc}}}
\nc{\bDelta}{\mathbf{\Delta}}
\nc{\Cech}{\textnormal{\v{C}}}
\nc{\Dperf}{\on{\mbf{D}}_{\mrm{perf}}}
\nc{\Coh}{\on{Coh}}
\nc{\Qcoh}{\on{Qcoh}}
\nc{\Dcoh}{\on{\mbf{D}}_{\mrm{coh}}}
\nc{\cl}{{\mrm{cl}}}
\nc{\Bl}{\on{Bl}}
\nc{\vir}{\mrm{vir}}
\nc{\CH}{\on{A}}
\nc{\et}{\mrm{\acute{e}t}}
\renc{\H}{\on{H}}
\nc{\BM}{\mrm{BM}}
\nc{\Z}{\bZ}
\nc{\Q}{\bQ}
\nc{\K}{{\on{K}}}
\nc{\KB}{\K^{\mrm{B}}}
\nc{\G}{{\on{G}}}
\nc{\KH}{\mrm{KH}}
\nc{\Ket}{\K^{\et}}
\nc{\KHet}{\KH^{\et}}
\nc{\Get}{\G^{\et}}
\nc{\rightrightrightarrows}
  {\mathrel{\substack{\textstyle\rightarrow\\[-0.6ex]
   \textstyle\rightarrow \\[-0.6ex]
   \textstyle\rightarrow}}}
\nc{\rightrightrightrightarrows}
  {\mathrel{\substack{\textstyle\rightarrow\\[-0.6ex]
   \textstyle\rightarrow \\[-0.6ex]
   \textstyle\rightarrow \\[-0.6ex]
   \textstyle\rightarrow}}}
\nc{\Einfty}{{\sE_\infty}}
\renc{\sp}{\mrm{sp}}
\nc{\Td}{\on{Td}}
\nc{\ch}{\on{ch}}
\nc{\RGamma}{\bR\Gamma}
\nc{\red}{\mrm{red}}
\nc{\der}{{\mrm{der}}}
\nc{\Mod}{{\mrm{Mod}}}
\nc{\Gr}{{\on{Gr}}}
\nc{\Ind}{\on{Ind}}
\nc{\form}{\widehat}
\nc{\R}{\bR}
\renc{\L}{\bL}
\nc{\otimesL}{\mathchoice{\overset{\bL}{\otimes}}{\otimes^\bL}{\otimes^\bL}{\otimes^\bL}}
\nc{\fibprodR}{\fibprod^\bR}
\nc{\uRHom}{\bR\uHom}
\nc{\GL}{\mrm{GL}}

\nc{\scr}{\term{derived commutative ring}}
\nc{\scrs}{\term{derived commutative rings}}

\nc{\inftyCat}{\term{$\infty$-category}}
\nc{\inftyCats}{\term{$\infty$-categories}}

\nc{\inftyGrpd}{\term{$\infty$-groupoid}}
\nc{\inftyGrpds}{\term{$\infty$-groupoids}}

\nc{\dA}{\term{derived Artin}}

\title{K-theory and G-theory of derived algebraic~stacks\vspace{-2mm}}
\author{Adeel~A.~Khan\vspace{-1mm}}
\date{2021-10-03}

\makeatletter
\def\l@subsection{\@tocline{2}{0pt}{4pc}{6pc}{}}
\makeatother

\begin{document}

\begin{abstract}
  These are some notes on the basic properties of algebraic K-theory and G-theory of derived algebraic spaces and stacks, and the theory of fundamental classes in this setting.
  \vspace{-5mm}
\end{abstract}

\maketitle

\renewcommand\contentsname{\vspace{-1cm}}
\tableofcontents

\setlength{\parindent}{0em}
\parskip 0.75em

\thispagestyle{empty}


\changelocaltocdepth{1}

\section*{Introduction}

  In \cite{KhanVirtual} I studied the cohomology and Borel--Moore homology of derived schemes and algebraic spaces, as well as Borel-type extensions to derived algebraic stacks.
  In these notes I describe the K-theoretic counterpart to that formalism.
  In this analogy, the K-theory of perfect complexes behaves like cohomology while G-theory (K-theory of coherent sheaves) behaves like Borel--Moore homology.
  This relationship can actually be made quite precise using the formalism of motivic categories: the theory of \emph{KGL-modules} provides a category of coefficients for K-theory\footnote{or rather its homotopy invariant version KH (see \ssecref{ssec:KH})}, and Borel--Moore homology in that setting (i.e., cohomology with coefficients in the dualizing complex) is G-theory.
  See \cite[\S 13.3]{CisinskiDegliseBook} and \cite{JinG}.

  The notes are almost entirely expository and were originally prepared as background material for the paper \cite{KhanVirtual}.
  I begin in Section~\ref{sec:qcoh} with some preliminaries on quasi-coherent complexes on derived stacks and various finiteness conditions.
  Thanks primarily to work of Hall and Rydh over the past several years, we can work with very general stacks.
  I also review some new compact generation results for stacks, obtained jointly with Ravi.

  In the \hyperref[sec:K]{second} and \hyperref[sec:G]{third} sections I define K-theory and G-theory and record their basic properties.
  Most of these go back to \cite{SGA6}, \cite{QuillenK}, and \cite{ThomasonTrobaugh} in the case of classical schemes, and to \cite{ToenGRR,ToenNotes,KrishnaRavi,HoyoisKrishna} in the case of classical stacks.

  The \hyperref[sec:KH]{fourth} section discusses the failure of homotopy invariance in K-theory of singular spaces and some of its ramifications.
  I explain how \emph{forcing} homotopy invariance for arbitrary spaces results in a new cohomology theory, first introduced by Weibel \cite{WeibelKH} for classical schemes, that agrees with K-theory on nonsingular spaces but behaves more ``correctly'' on singular ones.

  In the \hyperref[sec:Get]{fifth} section I explain how K-theory and G-theory acquire some useful descent properties after passage to rational coefficients.
  This leads to étale-local variants of K-theory and G-theory of stacks, which do not agree however with K-theory and G-theory even rationally.
  A subtlety here is the existence and behaviour of direct images in étale G-theory, which are not compatible with those in G-theory.
  Note that, for stacks, these étale-localized theories are the ones that are actually compatible with the étale motivic theories studied in \cite{KhanVirtual}.

  The most important aspect of the formalism in \cite{KhanVirtual} is the theory of fundamental classes.
  The K-theoretic counterpart to that part of the story, at least some of which is known to the experts, is developed in Section~\ref{sec:fund}.
  The K-theoretic virtual structure sheaf of \cite{Lee} and its basic properties fall out of this formalism for free.
  I also give another description of Gysin maps in G-theory using deformation to the normal stack.
  Then I recall some variants of the Grothendieck--Riemann--Roch theorem.
  One formulation, proven in \cite{KhanVirtual}, compares the virtual structure sheaf with the virtual fundamental class in Borel--Moore homology (or the Chow groups).
  Another is a direct generalization of the one proven in \cite{SGA6} and involves the behaviour of K-theoretic fundamental classes with respect to the $\gamma$-filtration.
  This inspires some unsolicited speculations on a theory of derived algebraic cycles.

  \subsection*{Conventions}

    I generally tried to follow the same conventions and notation as in \cite{KhanVirtual}.
    Since a substantial revision to that paper is in preparation, I was not always able to achieve this.

  \subsection*{Acknowledgments}

    Thanks to Mauro Porta and Tony Yue Yu for many questions and suggestions about \cite{KhanVirtual}, which eventually led to these notes being written.
    Thanks to Marc Hoyois, Charanya Ravi, and David Rydh for helpful conversations about stacks over the years.
    A second thanks to Charanya for pointing out many typos in a previous draft.

\changelocaltocdepth{2}

\section{Perfect and coherent complexes}
\label{sec:qcoh}

  \subsection{Complexes over derived commutative rings}

    Let $A$ be a derived commutative ring.
    We write $\D(A)$ for the derived \inftyCat of complexes over $A$, defined e.g. as in \cite[\S 25.2.1]{LurieSAG}, \cite[App.~A]{HalpernLeistnerPreygel}, or \cite[\S 2.4]{ToenVaquie}.

    We will say a complex $M \in \D(A)$ is \emph{connective}, resp. \emph{coconnective}, if we have $\pi_i(M) = 0$ for all $i < 0$, resp. for all $i > 0$.
    More generally we say $M$ is \emph{$n$-connective}, resp. \emph{$n$-coconnective}, if we have $\pi_i(M) = 0$ for all $i < n$, resp. for all $i > n$.
    We write
    \[
      \D(A)_{\ge n} \sub \D(A),
      \qquad \D(A)_{\le n} \sub \D(A),
    \]
    for the respective full subcategories.
    If $M$ is both connective and coconnective, i.e., if $\pi_i(M) = 0$ for all $i\ne 0$, then we say $M$ is \emph{discrete}.
    The full subcategory of discrete complexes
    \[ \D(A)^\heartsuit \sub \D(A)\]
    is equivalent to the abelian category $\Mod_{\pi_0(A)}$ of $\pi_0(A)$-modules.
    What we have just described is of course nothing else than the standard t-structure on $\D(A)$.

    The cohomologically oriented reader will want to write
    \[ \on{H}^{-i}(M) := \pi_i(M) \]
    and reset the notation as follows:
    \[
      \D(A)^{\le n} := \D(A)_{\ge -n}, 
      \qquad \D(A)^{\ge n} := \D(A)_{\le -n}.
    \]

    We recall some finiteness conditions on complexes.

    \begin{defn}
      A complex $M$ over $A$ is \emph{perfect} if it is in the thick subcategory of $\D(A)$ generated by $A$.
      That is, if it is built out of $A$ under finite (co)limits and direct summands.
    \end{defn}

    To define (pseudo)coherent complexes, it is convenient (but not necessary) to assume that $A$ is noetherian.
    Here this means that the ordinary commutative ring $\pi_0(A)$ is noetherian and that the homotopy groups $\pi_i(A)$ are finitely generated as $\pi_0(A)$-modules for all $i\ge 0$.
    Whenever we discuss (pseudo)coherent complexes below, the reader should either assume the ring is noetherian\footnote{
      Noetherianness can be replaced harmlessly by the slightly weaker property of coherence in the sense of \cite[Def.~7.2.4.13]{LurieHA}.
    } or should replace the definition of pseudocoherence below by \cite[Def.~7.2.4.10]{LurieHA}.

    \begin{defn}
      \begin{inlinedefnlist}
        \item
        A complex $M$ over $A$ is \emph{pseudocoherent}\footnote{
          In \cite{LurieHA,LurieSAG,HalpernLeistnerPreygel}, the term ``almost perfect'' is used instead of pseudocoherent.
        } if it is eventually connective, i.e., $\pi_i(M) = 0$ for $i \ll 0$, and $\pi_i(M)$ is finitely generated as a $\pi_0(A)$-module for all $i$.

        \item
        A complex $M$ over $A$ is \emph{coherent} if it is pseudocoherent and also eventually coconnective, i.e., $\pi_i(M) = 0$ for $i \gg 0$.
      \end{inlinedefnlist}
    \end{defn}

    \begin{defn}\label{defn:Tor}
      \begin{inlinedefnlist}
        \item
        A complex $M$ over $A$ is \emph{of Tor-amplitude $\le n$} if for every discrete complex $N \in \D(A)^\heartsuit$, the derived tensor product $M \otimes^\bL_{A} N$ is $n$-coconnective.
        Equivalently, for every coconnective complex $N \in \D(A)_{\le 0}$, the derived tensor product $M \otimes^\bL_{A} N$ is $n$-coconnective (see \cite[Prop.~7.2.4.23 (5)]{LurieHA}).
        
        \item
        A complex $M$ over $A$ is \emph{of finite Tor-amplitude} if it is of Tor-amplitude $\le n$ for some $n$.
      \end{inlinedefnlist}
    \end{defn}

    The following lemma summarizes the relationships between these finiteness conditions.

    \begin{lem}\label{lem:perfect is pseudocoherent}\leavevmode
      \begin{thmlist}
        \item
        The property of (pseudo)coherence is stable under finite (co)limits and direct summands.
        In other words, the (pseudo)coherent complexes form a thick subcategory of $\D(A)$.

        \item
        Every perfect complex over $A$ is pseudocoherent.

        \item\label{item:perfect is pseudocoherent/fTa}
        A pseudocoherent complex over $A$ is perfect if and only if it is of finite Tor-amplitude.

        \item
        If $A$ is eventually coconnective, then every perfect complex over $A$ is coherent.

        \item
        Assume $A$ is eventually coconnective and that every coherent complex is perfect.
        Then $A$ is discrete, and regular as an ordinary commutative ring.
      \end{thmlist}
    \end{lem}
    \begin{proof}
      \begin{inlinelist}[label={(\roman*)}]
        \item
        See \cite[Lem.~7.2.4.11 (1-2)]{LurieHA}.
        
        \item
        By definition of perfectness, it suffices to show that $A$ itself is pseudocoherent over $A$.
        If $A$ is noetherian, then this follows from the definition of pseudocoherence; see \cite[7.2.4.11 (3)]{LurieHA} for the non-noetherian case.

        \item
        See \cite[Prop.~7.2.4.23(4)]{LurieHA}.

        \item
        Again it suffices to show that $A$ is coherent over itself, which follows from the definitions when $A$ is eventually coconnective.

        \item
        See \cite[Lem.~11.3.3.3]{LurieSAG}.
      \end{inlinelist}
    \end{proof}

  \subsection{Quasi-coherent, perfect, and coherent complexes}

    Let $\sX$ be a \dA stack.
    We write $\Dqc(\sX)$ for the stable \inftyCat of \emph{quasi-coherent complexes} on $\sX$, see e.g. \cite[Subsect.~3.1]{ToenSurvey}.

    Over an affine derived scheme $\Spec(A)$, a quasi-coherent complex $\sF$ is the same datum as that of the complex $\RGamma(\Spec(A), \sF)$ over $A$, i.e., there is an equivalence
    \[ \Dqc(\Spec(A)) \simeq \D(A) \]
    given by $\sF \mapsto \RGamma(\Spec(A), \sF)$.
    We write $\sO_X \in \Dqc(\Spec(A))$ for the quasi-coherent complex whose derived global sections are $A \in \D(A)$.

    Over a \dA stack $\sX$, a quasi-coherent complex $\sF$ amounts to the data of quasi-coherent complexes $u^*(\sF) \in \Dqc(X)$ for every smooth morphism $u : X \to \sX$ with $X$ affine, together with a homotopy coherent system of compatibilities between these complexes as $u$ varies.
    For example, the structure sheaf $\sO_\sX \in \Dqc(\sX)$ satisfies $u^*(\sO_\sX) \simeq \sO_{X}$ for every $u$.

    \begin{defn}
      \begin{inlinedefnlist}
        \item
        A quasi-coherent complex $\sF$ on an affine derived scheme $X=\Spec(A)$ is \emph{perfect}, \emph{pseudocoherent}, or \emph{coherent} if $\RGamma(X, \sF)$ has the respective property as a complex over $A$.
        
        \item
        A quasi-coherent complex $\sF$ on a \dA stack $\sX$ is \emph{perfect}, \emph{pseudocoherent}, or \emph{coherent}, if $u^*(\sF)$ has the respective property for every smooth morphism $u : X \to \sX$ from an affine $X$.
      \end{inlinedefnlist}
      
      We write
      \[ \Dperf(\sX) \subseteq \Dqc(\sX) ~\text{and}~ \Dcoh(\sX) \subseteq \Dqc(\sX) \]
      for the full subcategories of perfect and coherent complexes on $\sX$.
    \end{defn}

    \begin{rem}
      If $\sX$ is a classical Artin stack, then $\Dqc(\sX)$ can be described as the derived \inftyCat of complexes of $\sO_\sX$-modules (on the lisse-étale topos) with quasi-coherent cohomology.
      See \cite[Prop.~1.3]{HallRydh}.
      In that language, $\Dcoh(\sX)$ is the full subcategory of complexes with bounded and coherent cohomology.
    \end{rem}

    \begin{defn}\label{defn:t-structure}
      Let $\sX$ be a \dA stack and $\sF \in \Dqc(\sX)$ a quasi-coherent complex.
      For an integer $n\in\Z$, we say that $\sF$ is \emph{$n$-connective} if $\pi_i(u^*\sF) = 0$ for all $i < n$ and smooth morphisms $u : X \to \sX$ with $X$ affine.
      We say $\sF$ is \emph{$n$-coconnective} if $\pi_i(u^*\sF) = 0$ for all $i > n$ and smooth morphisms $u : X \to \sX$ with $X$ affine.
      The ($0$-)connective and ($0$-)coconnective complexes define full subcategories
      \[ \Dqc(\sX)_{\ge 0} \sub \Dqc(\sX), \quad \Dqc(\sX)_{\le 0} \sub \Dqc(\sX), \]
      respectively, which together form a canonical t-structure on $\Dqc(\sX)$.
      This t-structure restricts to $\Dcoh(\sX)$, but typically not to $\Dperf(\sX)$ unless $\sX$ is a classical regular stack.
      We write
      \[ \Qcoh(\sX) := \Dqc(\sX)^\heartsuit, \quad \Coh(\sX) := \Dcoh(\sX)^\heartsuit \]
      for the hearts.
      Note that these are insensitive to the derived structure on $\sX$: they are equivalent to the abelian categories of quasi-coherent and coherent sheaves, respectively, on the classical truncation $\sX_\cl$.
    \end{defn}

    \begin{rem}\label{rem:Perf in Coh}
      Let $\sX$ be a \dA stack.
      Suppose that $\sX$ has \emph{bounded structure sheaf}, i.e., that the quasi-coherent complex $\sO_\sX$ is $n$-coconnective for some $n\gg 0$.
      In that case there is an inclusion $\Dperf(\sX) \subseteq \Dcoh(\sX)$.
      Indeed the properties of perfectness and coherence are both local by definition so this follows from \lemref{lem:perfect is pseudocoherent}.
    \end{rem}

    \begin{rem}\label{rem:tcomplt}
      The t-structure on $\Dqc(\sX)$ is left- and right-complete.
      That is, for every quasi-coherent complex $\sF \in \Dqc(\sX)$, the canonical morphisms
      \begin{align*}
        &\sF \to \lim_{n\ge 0} \tau_{\le n}(\sF),\\
        &\colim_{n\ge 0} \tau_{\ge -n}(\sF) \to \sF
      \end{align*}
      are invertible.
      Here $\tau_{\le n} : \Dqc(\sX) \to \Dqc(\sX)_{\le n}$ and $\tau_{\ge -n} : \Dqc(\sX) \to \Dqc(\sX)_{\le -n}$ are the truncation functors, left and right adjoint to the respective inclusions.
      This follows from the fact that $\Dqc(\sX)$ is a limit of stable \inftyCats with left- and right-complete t-structures, and t-exact transition functors (see \cite[Prop.~7.1.1.13]{LurieHA} and \cite[Chap.~3, Lem.~1.5.8]{GaitsgoryRozenblyum}).
    \end{rem}

  \subsection{Tensor, Hom, and inverse/direct images}

    We have the following basic operations on $\Dqc(\sX)$.

    \sssec{}

      For every \dA stack $\sX$, there is a (derived) tensor product $\otimesL$ on $\Dqc(\sX)$.
      It is left adjoint, as a bifunctor, to the internal Hom functor $\uRHom$.
      These are part of a closed symmetric monoidal structure on $\Dqc(\sX)$.

      For every fixed perfect complex $\sF \in \Dperf(\sX)$, the operation
      $$\sG \mapsto \sF \otimesL \sG$$
      preserves perfectness and coherence in $\Dqc(\sX)$.
      (One easily reduces to the case where $\sX = X$ is affine and $\sF = \sO_X$.)

      \begin{rem}\label{rem:a0p8sfh1}
        The symmetric monoidal structure allows the following characterization of perfect complexes over arbitrary $\sX$: namely, $\Dperf(\sX)$ consists of precisely those objects of $\Dqc(\sX)$ (or of $\Dcoh(\sX)$) which are \emph{dualizable} with respect to the symmetric monoidal structure.
        This follows by general nonsense (see \cite[Prop.~4.6.1.11]{LurieHA}) from the case of affines, which is a straightforward exercise (e.g. \cite[Prop.~6.2.6.2]{LurieSAG}).
      \end{rem}

    \sssec{}

      Given a morphism of \dA stacks $f : \sX \to \sY$, there is an inverse image functor
      \begin{equation*}
        \L f^* : \Dqc(\sY) \to \Dqc(\sX)
      \end{equation*}
      which is symmetric monoidal.
      It is right t-exact, i.e., preserves connectivity.
      It also preserves perfect complexes and pseudocoherent complexes.
      If $f$ is flat, then $\L f^* = f^*$ is also left t-exact.

    \sssec{}

      The inverse image functor admits a right adjoint, the direct image functor
      \begin{equation*}
        \R f_* : \Dqc(\sX) \to \Dqc(\sY).
      \end{equation*}
      By adjunction, it is left t-exact, i.e., preserves coconnectivity.
      If it is affine, then $\R f_* = f_*$ is also right t-exact.
      If $f$ is representable and proper (i.e., if it is representable, separated of finite type, and satisfies the valuative criterion), then $\R f_*$ preserves pseudocoherent and coherent complexes.
      See \cite[Thms.~5.6.0.2]{LurieSAG}, and \ssecref{ssec:coh proper} for a discussion of the non-representable case.

      \begin{rem}\label{rem:apdsfb}
        When $f$ is representable by qcqs derived algebraic spaces, $\R f_*$ commutes with colimits and satisfies base change and projection formulas.
        That is:
        \begin{defnlist}
          \item
          The functor $\L f^*$ is compact, i.e., its right adjoint $\R f_*$ preserves colimits.
          See \cite[Cor.~3.4.2.2 (2)]{LurieSAG}.

          \item
          For every homotopy cartesian square of \dA stacks
          \begin{equation*}
            \begin{tikzcd}
              \sX' \ar{r}{g}\ar{d}{p}
                & \sY' \ar{d}{q}
              \\
              \sX \ar{r}{f}
                & \sY,
            \end{tikzcd}
          \end{equation*}
          there is a canonical isomorphism
          \begin{equation*}
            \L q^*\R f_* \to \R g_* \L p^*
          \end{equation*}
          of functors $\Dqc(\sX) \to \Dqc(\sY')$.
          See \cite[Cor.~3.4.2.2 (3)]{LurieSAG}.
          
          \item
          For every $\sF \in \Dqc(\sX)$ and $\sG \in \Dqc(\sY)$, there is a canonical isomorphism
          \begin{equation*}
            \R f_*(\sF) \otimesL \sG \to \R f_*(\sF \otimesL \L f^*(\sG))
          \end{equation*}
          which is natural in $\sF$ and $\sG$.
          In other words, $\R f_*$ is $\Dqc(\sY)$-linear when $\Dqc(\sX)$ is regarded as a $\Dqc(\sY)$-module category via the symmetric monoidal functor $\L f^*$.
          See \cite[Rem.~3.4.2.6]{LurieSAG}.
        \end{defnlist}
        As we will see in the next subsection, representability can be considerably weakened.
      \end{rem}

  \subsection{Finiteness of Tor-amplitude}

    Recall that a morphism $A \to B$ of \scrs is of Tor-amplitude $\le n$ if $B$ is of Tor-amplitude $\le n$ as a complex over $A$ (\defnref{defn:Tor}).
    This condition is fpqc-local on the source and target, so we may extend to \dA stacks in the usual manner.
    For example, a morphism is of Tor-amplitude $\le 0$ if and only if it is flat.
    We also say a morphism is of \emph{finite Tor-amplitude} if it is of Tor-amplitude $\le n$ for some $n\ge 0$.
    We clearly have:

    \begin{prop}
      Let $f : \sX \to \sY$ be a morphism of \dA stacks.
      If $f$ is of Tor-amplitude $\le n$, then $\L f^*$ restricts to a functor
      \[ \L f^*: \Dqc(\sY)_{\le 0} \to \Dqc(\sX)_{\le n}, \]
      i.e., it sends coconnective complexes to $n$-coconnective complexes.
    \end{prop}

    \begin{rem}
      Recall that $\L f^*$ always preserves pseudocoherent complexes.
      Thus when $f$ is of finite Tor-amplitude, it also preserves coherent complexes.
    \end{rem}

    \begin{rem}
      If $f$ is proper representable, locally almost of finite presentation, and of finite Tor-amplitude, then $\R f_*$ preserves perfect complexes.
      See \cite[Thm.~6.1.3.2]{LurieSAG}.
    \end{rem}

    Recall that a morphism of \dA stacks $f : \sX \to \sY$ is \emph{quasi-smooth} if it is locally of finite presentation and the relative cotangent complex $\sL_{\sX/\sY}$ is of Tor-amplitude $\le 1$.
    Equivalently, it factors smooth-locally on the source as the inclusion $\sX \hook \sY'$ of the derived zero locus of a section of a vector bundle over a stack $\sY'$ which is smooth over $\sY$; see \cite[Prop.~2.3.14]{KhanRydh}.
    This is a derived version of the notion of \emph{local complete intersection} morphism.

    \begin{lem}\label{lem:qsm fTa}
      Let $f : \sX \to \sY$ be a quasi-smooth morphism of \dA stacks.
      Then $f$ is of finite Tor-amplitude.
    \end{lem}
    \begin{proof}
      If $f$ is smooth, then it is of Tor-amplitude $\le 0$.
      Therefore we may assume that $f$ is the inclusion of the derived zero locus $\sX = X$ of $n$ functions $g_1,\ldots,g_n$ on an affine derived scheme $\sY = Y$.
      We claim that $f$ is of Tor-amplitude $\le n$.
      By induction, we may assume that $n=1$.
      Then the exact triangle
      \begin{equation*}
        \sO_Y \xrightarrow{g_1} \sO_Y \to f_*(\sO_X)
      \end{equation*}
      shows that $f$ is of Tor-amplitude $\le 1$.
    \end{proof}

  \subsection{Finiteness of cohomological dimension}
  \label{ssec:finite coh dim}

    The functor $f_*$ is well-behaved for many non-representable morphisms as well.
    The relevant condition is as follows.

    \begin{defn}
      Let $f : \sX \to \sY$ be a morphism of \dA stacks.
      We say that $f$ is \emph{of cohomological dimension $\le n$} if $\R f_*$ restricts to a functor
      \[ \R f_* : \Dqc(\sX)_{\ge 0} \to \Dqc(\sY)_{\ge -n}, \]
      i.e., it sends connective complexes to $(-n)$-connective complexes.
      Equivalently, $\R f_*(\sF)$ is $(-n)$-connective for every \emph{discrete} $\sF \in \Qcoh(\sX)$ (see \cite[Lem.~A.1.6]{HalpernLeistnerPreygel}).
      We say that $f$ is \emph{of finite cohomological dimension} if it is of cohomological dimension $\le n$ for some integer $n\in\bN$.
    \end{defn}

    \begin{exam}
      We say that a \dA stack $\sX$ is of cohomological dimension $\le n$ if the structural morphism $f : \sX \to \Spec(\bZ)$ is.
      Note that this is the case if and only if the complex
      \[ \RGamma\big(\Spec(\bZ), \R f_*(\sF)\big) \simeq \RGamma(\sX, \sF) \]
      is $(-n)$-connective for all quasi-coherent sheaves $\sF \in \Qcoh(\sX)$.
      In other words, if
      $$\H^i(\sX, \sF) = \pi_{-i}\left(\RGamma(\sX, \sF)\right) = 0$$
      for all $i>n$ and all $\sF \in \Qcoh(\sX)$.
    \end{exam}

    \begin{defn}
      We say that a morphism $f : \sX \to \sY$ is \emph{universally of finite cohomological dimension} if for every morphism $\sY' \to \sY$ with $\sY'$ qcqs, the base change $f' : \sX \fibprod_\sY \sY' \to \sY'$ is of finite cohomological dimension.
    \end{defn}

    In \cite{HallRydh}, qcqs morphisms that are universally of finite cohomological dimension are called ``concentrated''.

    \begin{rem}
      If $\sY$ is quasi-compact with quasi-affine diagonal, then any qcqs morphism $f : \sX \to \sY$ of finite cohomological dimension is also \emph{universally} of finite cohomological dimension.
      See \cite[Lem.~2.5(v)]{HallRydh}.
    \end{rem}

    Our main interest in this condition is the good behaviour of $\R f_*$ it implies (compare \remref{rem:apdsfb}):

    \begin{thm}\label{thm:uni fin coh dim}
      Let $f : \sX \to \sY$ be a quasi-compact quasi-separated morphism of \dA stacks.
      If $f$ is universally of finite cohomological dimension, then $\R f_*$ commutes with colimits and satisfies the base change and projection formulas (as in \remref{rem:apdsfb}).
    \end{thm}
    \begin{proof}
      See \cite[Thm.~2.6]{HallRydh} and \cite[Prop.~A.1.5]{HalpernLeistnerPreygel}.
    \end{proof}

    \begin{rem}\label{rem:afp80h1}
      The fact that $\R f_*$ commutes with colimits (i.e., that $\L f^*$ is a compact functor) implies by adjunction that $\L f^*$ preserves compact objects.
      In particular, \thmref{thm:uni fin coh dim} implies that if $\sX$ is a qcqs \dA stack which is universally of finite cohomological dimension, then the structure sheaf $\sO_\sX$ is compact as an object of $\Dqc(\sX)$.
    \end{rem}

    \begin{rem}\label{rem:ophunas}
      A morphism $f : \sX \to \sY$ is of finite cohomological dimension if and only if the underlying morphism $f_\cl : \sX_\cl \to \sY_\cl$ of classical truncations is of finite cohomological dimension.
      Indeed, this follows by consideration of the following commutative square
      \[ \begin{tikzcd}
        \Qcoh(\sX_\cl)\ar{r}{\R f_{\cl,*}}\ar{d}{i_{\sX,*}}
        & \D(\sY_\cl)_{\le 0}\ar{d}{i_{\sY,*}}
        \\
        \Qcoh(\sX)\ar{r}{\R f_*}
        & \D(\sY)_{\le 0},
      \end{tikzcd} \]
      where the vertical arrows are direct image along the inclusions of the classical truncations.
      Recall that $i_{\sX,*}$ and $i_{\sY,*}$ are t-exact (since $i_\sX$ and $i_\sY$ are affine morphisms) and that the left-hand arrow is an equivalence (\defnref{defn:t-structure}).
    \end{rem}

    \begin{exam}\label{exam:a-psdf8}
      Let $f : \sX \to \sY$ be a morphism of \dA stacks.
      If $f$ is representable by qcqs derived algebraic spaces, then it is universally of finite cohomological dimension.
      See \cite[Prop.~2.5.4.4, Thm.~3.4.2.1]{LurieSAG}.
    \end{exam}

    \begin{exam}\label{exam:apsdufgp01}
      Let $G$ be an affine group scheme of finite presentation over an affine scheme $S$ and consider the morphism $f : BG \to S$, the structural morphism of the classifying stack.
      \begin{defnlist}
        \item
        If $G$ is linearly reductive, then $f$ is universally of cohomological dimension zero.
        
        \item
        If $S$ is the spectrum of a field of characteristic zero, then $f$ is universally of finite cohomological dimension (see \cite[Thm.~1.4]{HallRydhGroups}).
      \end{defnlist}
    \end{exam}

    \begin{exam}\label{exam:finite cohomological dimension}
      Let $\sX$ be a qcqs derived $1$-Artin stack.
      Then $\sX$ is of finite cohomological dimension under any one of the following conditions:
      \begin{defnlist}
        \item
        $\sX$ is of characteristic zero and its stabilizers at all points are affine.

        \item
        $\sX$ is of positive characteristic and its stabilizers at all points are linearly reductive.

        \item
        $\sX$ is of mixed characteristic and its stabilizers at all points are ``nice'' (extensions of tame finite étale groups by groups of multiplicative type).

        \item
        $\sX$ is of mixed characteristic, its (classical) inertia stack is of finite presentation over $\sX_\cl$, its stabilizers at all points are affine, and its stabilizers at points of positive characteristic are linearly reductive.
      \end{defnlist}
      Similarly, a morphism of qcqs \dA stacks $f : \sX \to \sY$ is of finite cohomological dimension if its fibres $\sX \fibprod_\sY Y$ satisfy one of the above conditions, for all morphisms $v : Y \to \sY$ with $Y$ affine.
      See \cite[Thm.~1.4.2]{DrinfeldGaitsgoryFiniteness} and \cite[Thm.~2.1]{HallRydhGroups}.
    \end{exam}

  \subsection{Cohomological properness}
  \label{ssec:coh proper}

    We would now like to study the question of when the direct image functor preserves coherent complexes.

    \begin{defn}\label{defn:coh proper}
      Let $f : \sX \to \sY$ be a morphism of \dA stacks.
      Assume that $f$ is of finite cohomological dimension.
      We say that $f$ is \emph{cohomologically proper} if for every coherent sheaf $\sF \in \Coh(\sX)$, the direct image $\R f_*(\sF) \in \Dqc(\sY)$ is a coherent complex.
    \end{defn}

    \begin{rem}
      Since $f$ is of finite cohomological dimension, $\R f_*(\sF) \in \Dqc(\sY)$ is bounded (cohomologically) for any coherent sheaf $\sF \in \Coh(\sX)$.
      Thus the condition of \defnref{defn:coh proper} is that it has coherent cohomologies, i.e., that for every smooth morphism $u : Y \to \sY$ with $Y$ affine, the groups
       \[ \H^i(Y, u^* \R f_*(\sF)) = \pi_{-i}(\RGamma(Y, u^* \R f_*(\sF))) \]
      are finitely generated over $\H^0(Y, \sO_Y)$ for all $i$.
    \end{rem}

    \begin{rem}
      If $f : \sX \to \sY$ is cohomologically proper, then any coherent \emph{complex} $\sF \in \Dcoh(\sX)$ also has coherent direct image $\R f_*(\sF) \in \Dcoh(\sY)$.
      This follows by induction on the Postnikov tower, i.e., using the exact triangles
      \[ \pi_n(\sF)[n] \to \tau_{\le n}(\sF) \to \tau_{\le n-1}(\sF) \]
      coming from the (bounded) t-structure on $\Dcoh(\sX)$.
    \end{rem}

    \begin{rem}
      As in \remref{rem:ophunas}, the condition of cohomological properness may be checked on classical truncations.
    \end{rem}

    \begin{prop}
      Let $f : \sX \to \sY$ be a morphism of finite cohomological dimension between noetherian derived $1$-Artin stacks.
      If $f$ is proper (of finite type, separated and satisfying the valuative criterion), then it is cohomologically proper.
    \end{prop}
    \begin{proof}
      See \cite[Thm.~1]{Faltings}, \cite[Thm.~1.2]{Olsson}.
    \end{proof}

    \begin{prop}
      Let $f : \sX \to \sY$ be a cohomologically proper morphism of \dA stacks that is universally of finite cohomological dimension.
      If $f$ is moreover almost of finite presentation and of finite Tor-amplitude, then $\R f_*$ also preserves perfect complexes.
    \end{prop}
    \begin{proof}
      Since $f$ is universally of finite cohomological dimension, formation of $\R f_*$ is stable under base change (\thmref{thm:uni fin coh dim}).
      Since perfectness is smooth-local by definition, we may assume that $\sY = Y$ is affine.
      By \lemref{lem:perfect is pseudocoherent}\ref{item:perfect is pseudocoherent/fTa} it will suffice to show that the coherent complex $\R f_*(\sF) \in \Dcoh(Y)$ is of finite Tor-amplitude for every perfect complex $\sF \in \Dperf(\sX)$.
      Now the claim follows from \cite[Prop.~6.1.3.1]{LurieSAG}.
    \end{proof}

  \subsection{The resolution property}
  \label{ssec:res}

    We make a short digression to record a very useful property of stacks, studied in detail in \cite{ThomasonHilbert,TotaroRes,Gross}.
    For our purposes it will also be useful to consider a stronger  ``derived'' variant of the definition.

    \begin{defn}
      Let $\sX$ be a qcqs \dA stack.
      \begin{defnlist}
        \item
        We say that $\sX$ has the \emph{resolution property} if, for every quasi-coherent sheaf $\sF \in \Qcoh(\sX)$, there exists a small collection $\{\sE_\alpha\}_{\alpha}$ of finite locally free sheaves on $\sX$ and a morphism
        \[ \phi : \bigoplus_\alpha \sE_\alpha \twoheadrightarrow \sF \]
        which is surjective (on $\pi_0$).
        
        \item
        We say that $\sX$ has the \emph{derived resolution property} if the above condition holds for every quasi-coherent \emph{complex} $\sF \in \Dqc(\sX)$.
      \end{defnlist}
    \end{defn}

    \begin{rem}
      It suffices to check the derived resolution property for connective complexes $\sF \in \Dqc(\sX)_{\ge 0}$ only, since the connective cover $\tau_{\ge 0}(\sF) \to \sF$ is bijective on $\pi_0$ for any $\sF \in \Dqc(\sX)$.
    \end{rem}

    \begin{rem}
      If $\sX$ admits the (resp. derived) resolution property and is quasi-compact, and $\sF \in \Qcoh(\sX)$ is of finite type (resp. $\sF \in \Dqc(\sX)$ has $\pi_0(\sF)$ of finite type), then there exists a surjection $\sE \twoheadrightarrow \sF$ where $\sE$ is a single finite locally free sheaf.
    \end{rem}

    \begin{prop}\label{prop:ou0pbasp0b1}
      Let $f : \sX \to \sY$ is a quasi-projective morphism of \dA stacks.
      If $\sY$ has the (resp. derived) resolution property, then so does $\sX$.
    \end{prop}
    \begin{proof}
      See \cite[Lem.~A.8]{KhanRavi} for the usual resolution property; the argument for the derived version is the same.
    \end{proof}

    \begin{exam}\label{exam:i9ytfgsa91}
      Let $G$ be a \emph{linearly reductive} affine group scheme over an affine scheme $S$.
      If $BG$ has the resolution property (e.g., if $G$ is embeddable \cite[Rem.~2.5]{AlperHallRydh}), then it also has the derived resolution property.
      Indeed, let $\sF \in \Dqc(BG)_{\ge 0}$ be a connective quasi-coherent complex.
      By the resolution property there exists a surjection $\phi_0 : \bigoplus_\alpha \sE_\alpha \twoheadrightarrow \pi_0(\sF)$ for some small collection of finite locally free sheaves $\sE_\alpha$ on $BG$.
      Note that each $\sE_\alpha$ is projective as an object of $\Dqc(BG)_{\ge 0}$, since the underlying object of $\Dqc(S)_{\ge 0}$ is projective and $BG$ is cohomologically affine (see e.g. \cite[Lem.~2.17]{HoyoisEquivariant}).
      Therefore, there exists a lift $\phi : \bigoplus_\alpha \sE_\alpha\twoheadrightarrow \sF$ of $\phi_0$, still surjective on $\pi_0$.
    \end{exam}

    \begin{exam}\label{exam:0a7s7h01}
      Let $S$ be the spectrum of a field of characteristic zero and let $G$ be an affine group scheme of finite type over $S$.
      Then $BG$ has the derived resolution property.
      Indeed, $G$ is a closed subgroup of $\GL_{n,S}$ for some $n$.
      Since the latter is linearly reductive (since $S$ is of characteristic zero), \examref{exam:i9ytfgsa91} implies that $B\GL_n$ has the derived resolution property.
      By \propref{prop:ou0pbasp0b1} it will therefore suffice to show that the morphism $BG \to B\GL_{n,S}$ is quasi-projective.
      Indeed, $[\GL_{n,S}/G]$ admits a $\GL_{n,S}$-equivariant immersion $[\GL_{n,S}/G] \hook \bP_S(\sE)$ over $S$ for some finite locally free sheaf $\sE$ on $BG$ (see the proof of \cite[Cor.~5.5.6]{Springer}), so that $BG \to B\GL_{n,S}$ is the quotient by $\GL_{n,S}$ of the quasi-projective morphism $[\GL_{n,S}/G] \to S$.
    \end{exam}

  \subsection{Perfect stacks}
  \label{ssec:perfect}
    
    \begin{defn}
      Let $\sX$ be a \dA stack.
      We say that $\sX$ is \emph{perfect} if the canonical functor
      $$\Ind(\Dperf(\sX)) \to \Dqc(\sX)$$
      is an equivalence.
    \end{defn}
    
    In other words, $\sX$ is perfect if $\Dqc(\sX)$ is compactly generated, and the compact objects are precisely the perfect complexes.
    See \cite[Chap.~1, \S 7]{GaitsgoryRozenblyum} for background on compactly generated stable \inftyCats and ind-completion.

    \begin{rem}
      Various definitions of perfectness can be found in the literature.
      For example, \cite[Def.~3.2]{BenZviFrancisNadler} and \cite[Def.~9.4.4.1]{LurieSAG} require the further condition that $\sX$ has affine or quasi-affine diagonal, respectively.
    \end{rem}

    \begin{thm}\label{thm:nice}
      Let $\sX$ be a qcqs derived $1$-Artin stack whose stabilizer at every point is nice (i.e., an extension of a finite étale group scheme, of order prime to the residue characteristics of $S$, by a multiplicative type group scheme).
      Then $\sX$ is perfect.
    \end{thm}

    This result was proven jointly with C.~Ravi, and is stated in \cite[Thm.~2.24]{KhanRavi} under the assumption that $\sX$ has separated diagonal.
    This assumption was not used in the proof, which we will review below in \ssecref{ssec:perfcrit}.

    \begin{exam}
      Here are some special cases of \thmref{thm:nice}:
      \begin{defnlist}
        \item
        \emph{Derived algebraic spaces.}
        Every qcqs derived algebraic space is perfect.
        This case was first proven by To\"en \cite{ToenAzumaya} (see also \cite[Prop.~9.6.1.1]{LurieSAG}), building on previous work like \cite{ThomasonTrobaugh} and \cite[Thm.~3.1.1]{BondalVdB} for classical schemes.

        \item
        \emph{Quotient stacks.}
        Quotients of qcqs derived algebraic spaces by nice group schemes have nice stabilizers (since subgroups of nice groups are nice), hence are perfect.
        This was proven in \cite[Prop.~3.3]{KrishnaRavi} (based on \cite{HallRydh,HallRydhGroups}).

        \item
        \emph{Deligne--Mumford stacks and tame Artin stacks.}
        Any \emph{tame} qcqs derived Deligne--Mumford stack is perfect.
        (Recall that tameness means that the stabilizers at all points, which are finite by assumption, are of order coprime to the residue characteristics of the base and hence nice.)
        This was proven in \cite[Thm.~4.7]{ToenAzumaya} in the case of finite inertia, and in the classical case it follows by combining \cite[Thm.~A]{HallRydh} and \cite[Thm.~C]{HallRydhGroups}.
        In fact, the latter shows more generally that tame Artin stacks in the sense of \cite{AbramovichOlssonVistoli} are perfect.

        \item
        \emph{Classical Artin stacks with nice stabilizers.}
        The special case of \thmref{thm:nice} for $\sX$ classical was proven under additional assumptions in \cite[Thm.~5.1]{AlperHallRydhLuna}, \cite[Prop.~14.1]{AlperHallRydh}, and \cite[Cor.~3.17]{HoyoisKrishna} (the latter based on \cite{AlperHallHalpernLeistnerRydh} and \cite[Thm.~C]{HallRydh}).
      \end{defnlist}
    \end{exam}

    We also have the following result for quotient stacks:

    \begin{thm}\label{thm:qfund}
      Let $\sX = [X/G]$ be the quotient of a quasi-projective derived scheme $X$ by a linear action of an affine group scheme $G$ of finite presentation over an affine scheme $S$.
      If $S$ is the spectrum of a field of characteristic zero, or if $G$ is linearly reductive and embeddable, then $\sX$ is perfect.
    \end{thm}

    These cases were first proven in \cite[Cor.~3.22]{BenZviFrancisNadler} and \cite[Thm.~2.24, A.9.1(a)]{KhanRavi}, respectively.



  \ssec{Perfectness criteria}
  \label{ssec:perfcrit}

    In this section we give some criteria to check perfectness of stacks, most notably \thmref{thm:perfect scallop}.
    We then use these to prove Theorems~\ref{thm:nice} and \ref{thm:qfund}.
    First note the following:

    \begin{lem}\label{lem:a-shg1123}
      Let $\sX$ be a \dA stack.
      \begin{thmlist}
        \item\label{item:a-shg1123/=>}
        If $\sX$ is quasi-separated then every compact object of $\Dqc(\sX)$ is a perfect complex.
        
        \item\label{item:a-shg1123/<=}
        If $\sO_\sX \in \Dqc(\sX)$ is compact, then every perfect complex on $\sX$ is a compact object of $\Dqc(\sX)$.
      \end{thmlist}
    \end{lem}
    \begin{proof}
      For \itemref{item:a-shg1123/=>}, let $\sF \in \Dqc(\sX)$ be a compact object.
      It is enough to show that $u^*(\sF) \in \Dqc(X)$ is perfect for every smooth morphism $u : X \to \sX$ with $X$ affine.
      Since $\sX$ is quasi-separated, $u$ is qcqs (and representable), so $u^*$ is a compact functor (\thmref{thm:uni fin coh dim}, \examref{exam:a-psdf8}).
      In particular, $u^*(\sF)$ is compact and hence also perfect, since $X$ is affine.

      For \itemref{item:a-shg1123/<=}, note that for any $\sF \in \Dperf(\sX)$ we have
      \[
        \Maps_{\Dqc(\sX)}(\sF, -)
        \simeq \Maps_{\Dqc(\sX)}(\sO_\sX, \sF^\vee \otimesL -),
      \]
      where $\sF^\vee = \uRHom(\sF, \sO_\sX)$ is the derived dual (see \remref{rem:a0p8sfh1}).
      Since $\otimes^\bL$ commutes with colimits in each argument and $\sO_\sX$ is compact, it follows that this functor commutes with colimits, i.e., that $\sF$ is compact.
    \end{proof}

    \begin{lem}\label{lem:af0p-h11}
      Let $\sX$ be a qcqs \dA stack.
      Then $\sX$ is perfect if and only if the following conditions hold:
      \begin{thmlist}
        \item
        The structure sheaf $\sO_\sX$ is a compact object of $\Dqc(\sX)$.

        \item\label{item:af0p-h11/pu0ug7}
        For every nonzero quasi-coherent complex $\sF \in \Dqc(\sX)$, there exists a perfect complex $\sE \in \Dperf(\sX)$ and a nonzero morphism $\phi : \sE \to \sF$.
      \end{thmlist}
    \end{lem}
    \begin{proof}
      By definition, $\sX$ is perfect if and only if $\Dqc(\sX)$ is compactly generated and its compact objects are precisely the perfect complexes.
      The latter holds precisely when $\sO_\sX \in \Dqc(\sX)$ is compact (one direction is \lemref{lem:a-shg1123}\itemref{item:a-shg1123/<=}, the other is clear since $\sO_\sX$ is always perfect).
      Compact generation is equivalent by \cite[Cor.~C.6.3.3]{LurieSAG} (which applies because $\Dqc(\sX)$ is presentable) to the assertion that for every nonzero $\sF \in \Dqc(\sX)$, there exists a compact object $\sE \in \Dperf(\sX)$ and a nonzero morphism $\phi : \sE \to \sF$.
      The claim follows.
    \end{proof}

    \begin{cor}\label{cor:as0fg1}
      Let $\sX$ be a qcqs \dA stack.
      If $\sO_\sX \in \Dqc(\sX)$ is compact (e.g. $\sX$ is universally of finite cohomological dimension) and $\sX$ has the derived resolution property, then $\sX$ is perfect.
    \end{cor}
    \begin{proof}
      Let us check the condition \lemref{lem:af0p-h11}\itemref{item:af0p-h11/pu0ug7}.
      Let $\sF \in \Dqc(\sX)$ be a nonzero quasi-coherent complex.
      Replacing $\sF$ by a shift if necessary, we may assume that $\pi_0(\sF) \ne 0$.
      Then the claim follows immediately from the derived resolution property.
    \end{proof}

    \begin{prop}\label{prop:dsyg9o1}
      Let $f : \sX \to \sY$ be a quasi-affine (resp. quasi-projective) morphism of qcqs \dA stacks.
      If $\sY$ is perfect (resp. and has the derived resolution property), then $\sX$ is perfect.
    \end{prop}
    \begin{proof}
      Since $f$ is universally of finite cohomological dimension (\examref{exam:a-psdf8}), it follows from \remref{rem:afp80h1} that $\L f^*$ sends the compact object $\sO_\sY \in \Dqc(\sY)$ to the compact object $\sO_\sX \in \Dqc(\sX)$.
      It is therefore enough to show that $\bL f^* : \Dqc(\sY) \to \Dqc(\sX)$ generates under colimits.
      In the quasi-affine case this is clear.
      In the quasi-projective case we may reduce (when $\sY$ has the derived resolution property) to the case where $\sX$ is a projective bundle over $\sY$, which follows from \cite[Thm.~3.3]{KhanKblow}.
      Compare \cite[A.9.1(a)]{KhanRavi}.
    \end{proof}

    We now move to the formulation of \thmref{thm:perfect scallop}.
    We will require some preliminary definitions.

    \begin{defn}\label{defn:0pg10u70g}
      We say that a \dA stack $\sX$ is \emph{absolutely perfect} if it is perfect and moreover the canonical functor
      $$\Ind(\Dperf(\sX~\mrm{on}~Z)) \to \Dqc(\sX~\mrm{on}~Z)$$
      is an equivalence for every cocompact closed subset $Z \sub \abs{\sX}$.
      Here $\Dqc(\sX~\mrm{on}~Z)$ is the full subcategory of $\Dqc(\sX)$ spanned by quasi-coherent complexes $\sF$ supported on $Z$ (i.e., $j^*(\sF) = 0$ for $j : \sX \setminus Z \hook \sX$ the complementary open immersion), and similarly for $\Dperf(\sX~\mrm{on}~Z)$.
    \end{defn}

    In many cases, perfection automatically implies absolute perfection:

    \begin{lem}\label{lem:sac9yg}
      Let $\sX$ be a qcqs \dA stack.
      If $\sX$ admits the resolution property, then it is perfect if and only if it is absolutely perfect.
    \end{lem}
    \begin{proof}
      Let $Z \sub \abs{\sX}$ be a cocompact closed subset.
      By \cite[Prop.~8.2]{RydhApproxSheaves} there exists an ideal $\sI \sub \sO_{\sX_\cl} = \pi_0(\sO_\sX)$ of finite type whose vanishing locus is $Z$.
      Since $\sX$ admits the resolution property, there exists a surjection $\sE \twoheadrightarrow \sI$ where $\sE$ is a finite locally free sheaf on $\sX$.
      Let $\sK \in \Dqc(\sX)$ denote the Koszul complex of the resulting cosection $\sigma : \sE \twoheadrightarrow \sI \sub \sO_\sX$, i.e., the structure sheaf of its derived zero locus.
      Note that $\sK$ is a perfect complex supported on $Z$ by construction, and hence gives rise to a functor
      \[ \sK \otimes^\bL - : \Dqc(\sX) \to \Dqc(\sX~\mrm{on}~Z) \]
      which sends perfect complexes on $\sX$ to perfect complexes supported on $Z$.
      It will suffice to show that this functor generates under colimits, or equivalently that its right adjoint $\uRHom(\sK, -) \simeq \sK^\vee \otimes -$ is conservative.
      Since $\sX$ is quasi-compact, we may choose a smooth surjection $u : X \twoheadrightarrow \sX$ with $X$ affine.
      Then $u^* : \Dqc(\sX) \to \Dqc(X)$ is conservative, symmetric monoidal, and sends $\sK$ to the Koszul complex of $f^*(\sigma) : f^*(\sE) \to \sO_X$.
      Thus we may assume that $\sX = X$ is affine, in which case it is enough to show that $\sK$ generates $\Dqc(X~\mrm{on}~Z)$ under colimits, which is proven e.g. in \cite[Lem.~4.10]{ToenAzumaya}.
    \end{proof}

    Recall that an \emph{étale neighbourhood} of a cocompact closed subset $Z \sub \abs{\sX}$ is an étale morphism $f : \sX' \to \sX$ which restricts to an isomorphism $f^{-1}(Z)_\red \to Z_\red$.
    The following, taken from \cite[\S 2.3]{KhanRavi}, is a more flexible variant of \cite[Def.~2.5.3.1]{LurieSAG}.

    \begin{defn}\label{defn:scallop}
      Let $\sX$ be a quasi-compact derived Artin stack.
      A \emph{scallop decomposition} $(\sU_i, \sV_i, u_i)_i$ of $\sX$ is a finite filtration by quasi-compact open substacks
      \[
        \initial = \sU_0
        \hook \sU_1
        \hook \cdots
        \hook \sU_n = \sX,
      \]
      together with étale neighbourhoods $u_i : \sV_i \to \sU_i$ of $\sU_i \setminus \sU_{i-1}$.
    \end{defn}

    We can now state the criterion, which is abstracted out of the proof of \cite[Thm.~2.24]{KhanRavi}:

    \begin{thm}\label{thm:perfect scallop}
      Let $\sX$ be a qcqs \dA stack.
      Suppose given a scallop decomposition $(\sU_i, \sV_i, u_i)_i$ of $\sX$.
      If each $\sV_i$ is absolutely perfect (resp. each $\sO_{\sV_i}$ is compact), then $\sX$ is perfect (resp. $\sO_\sX$ is compact).
    \end{thm}

    \begin{proof}
      One argues exactly as in the proof of \cite[Thm.~2.24]{KhanRavi} (which itself is adapted from arguments of \cite{ThomasonTrobaugh} and \cite{BondalVdB}).
      Namely, one first shows as in \cite[Lem.~2.27]{KhanRavi} that $\sX$ the compact objects of $\Dqc(\sX)$ are the perfect complexes.
      Then one proceeds by an induction on the scallop decomposition to glue together a set of perfect complexes generating $\Dqc(\sX)$, using descent by étale neighbourhoods for $\Dqc(-)$.
    \end{proof}

    \begin{rem}
      In fact a small modification of the proof shows that $\sX$ is even absolutely perfect, but we will not need this.
    \end{rem}

    Finally we apply these results to prove Theorems~\ref{thm:nice} and \ref{thm:qfund}.

    \begin{proof}[Proof of \thmref{thm:qfund}]
      By Examples~\ref{exam:i9ytfgsa91} and \ref{exam:0a7s7h01}, $BG$ has the derived resolution property.
      By \examref{exam:apsdufgp01} it is also universally of finite cohomological dimension (\remref{rem:afp80h1}).
      By \propref{prop:ou0pbasp0b1} and \examref{exam:a-psdf8}, the same then holds for $\sX = [X/G]$.
      Thus $\sX$ is perfect by \corref{cor:as0fg1}.
    \end{proof}

    \begin{proof}[Proof of \thmref{thm:nice}]
      Suppose that $\sX$ has nice stabilizers.
      By the local structure theorem of \cite{AlperHallHalpernLeistnerRydh} (cf. \cite[Thm.~2.12(ii)]{KhanRavi}), $\sX$ admits a scallop decomposition $(\sU_i, \sV_i, u_i)_i$ where each $\sV_i$ is a quotient of a quasi-affine derived scheme by an action of a linearly reductive embeddable group scheme over an affine base.
      By \thmref{thm:qfund}, each $\sV_i$ is perfect.
      Since it also admits the resolution property (\propref{prop:ou0pbasp0b1} and \examref{exam:i9ytfgsa91}), it is moreover absolutely perfect by \lemref{lem:sac9yg}.
      Now the claim follows from \thmref{thm:perfect scallop}.
    \end{proof}




  \subsection{Continuity}

    The following is a generalization of Thomason's \cite[Prop.~3.20]{ThomasonTrobaugh} from the case of qcqs schemes.

    \begin{thm}\label{thm:contin}
      Let $(\sX_\alpha)_\alpha$ be a cofiltered system of \dA stacks with affine transition morphisms.
      Denote by $\sX$ the limit and by $p_\alpha : \sX \to \sX_\alpha$ the projections.
      Then we have:
      \begin{thmlist}
        \item
        The inverse image functors $p_\alpha^*$ induce a canonical equivalence of stable presentable \inftyCats
        \[
          \colim_\alpha \Dqc(\sX_\alpha)
          \to \Dqc(\sX),
        \]
        where the colimit is taken in the (very large) \inftyCat of presentable \inftyCats and colimit-preserving functors.

        \item\label{item:contin/perf}
        If $\sX$ and each $\sX_\alpha$ are perfect, then $p_\alpha^*$ also induce an equivalence
        \[
          \colim_\alpha \Dperf(\sX_\alpha)
          \to \Dperf(\sX),
        \]
        where the colimit is taken in the (large) \inftyCat of small \inftyCats.
      \end{thmlist}
    \end{thm}
    \begin{proof}
      When $\sX$ and $\sX_\alpha$ are perfect, the second equivalence is obtained by passage to compact objects from the first (see \cite[Prop.~5.5.7.8, Cor. 4.4.5.21]{LurieHTT}).

      The first statement is equivalent by \cite[Cor.~5.5.3.4, Thm.~5.5.3.18]{LurieHTT} to the assertion that the functor
      \begin{equation}\label{eq:afsd0ph1}
        \Dqc(\sX) \to \lim_\alpha \Dqc(\sX_\alpha),
      \end{equation}
      induced by the direct image functors $p_{\alpha,*}$, is an equivalence (where the limit is taken in the \inftyCat of large \inftyCats).

      Without loss of generality we may assume that the indexing category admits a terminal object $0$.
      Let $T_0$ be an affine derived scheme and $t_{0} : T_0 \to \sX_{0}$ a smooth morphism, and consider the analogous functor
      \[ \Dqc(T) \to \lim_{\alpha\ge\alpha_0} \Dqc(T_\alpha) \]
      where $T_\alpha = T_{0} \fibprod_{\sX_{\alpha_0}} \sX_\alpha$ and $T = T_{0} \fibprod_{\sX_{\alpha_0}} \sX$.
      Since the transition maps $\sX_\alpha\to\sX_{\alpha_0}$ are affine, $T$ and each $T_\alpha$ is affine.
      Therefore by \cite[Cor.~4.5.1.5]{LurieSAG} this functor is an equivalence.

      Since the transition maps are affine, the base change formula (\remref{rem:apdsfb}) implies that the diagram
      \[\begin{tikzcd}
        \Dqc(\sX) \ar{r}\ar{d}
        & \cdots \ar{r}
        & \Dqc(\sX_\alpha) \ar{r}\ar{d}
        & \Dqc(\sX_\beta) \ar{r}\ar{d}
        & \cdots \ar{r}
        & \Dqc(\sX_0) \ar{d}
        \\
        \Dqc(T) \ar{r}
        & \cdots \ar{r}
        & \Dqc(T_\alpha) \ar{r}
        & \Dqc(T_\beta) \ar{r}
        & \cdots \ar{r}
        & \Dqc(T_0)
      \end{tikzcd}\]
      commutes.
      Therefore, the functor \eqref{eq:afsd0ph1} is the limit over the analogous functors for every pair $(T_0,t_0)$, which are all equivalences by above.
    \end{proof}

    \begin{cor}\label{cor:continlocfr}
      Let $\sX$ be a perfect \dA stack.
      Suppose that $\sX$ is the limit of a cofiltered system $(\sX_\alpha)_\alpha$ of perfect \dA stacks with affine transition morphisms.
      Then the equivalence of \thmref{thm:contin}\itemref{item:contin/perf} restricts to an equivalence
      \[
        \colim_\alpha \Dperf(\sX_\alpha)_{\mrm{locfr}}
        \to \Dperf(\sX)_{\mrm{locfr}}
      \]
      where $(-)_{\mrm{locfr}}$ denotes the full subcategory of locally free sheaves of finite rank.
    \end{cor}
    \begin{proof}
      Since the property is smooth-local, we may reduce as in the proof of \thmref{thm:contin} to the case where $\sX=X$ and $\sX_\alpha=X_\alpha$ are affine.
      Choose an index $0$ and let $\sF_0 \in \Dperf(X_0)$, $\sF_\alpha = \sF_0|_{X_\alpha} \in \Dperf(X_\alpha)$ for every $\alpha$, and $\sF = \sF_0|_{X} \in \Dperf(X)$.
      Since inverse image preserves finite locally free sheaves, it is enough to show that if $\sF$ is finite locally free, then so is $\sF_\alpha$ for some $\alpha$.
      For every $\alpha$ let $U_\alpha \sub X_\alpha$ denote the open locus where $\sF_\alpha$ is finite locally free (see \cite[Prop.~2.9.3.2]{LurieSAG}).
      By construction we have $U_\alpha = U_0 \fibprod_{X_0} X_\alpha$ for all $\alpha$ and $U = U_0 \fibprod_{X_0} X$.
      As an open subscheme of $X_0 = \Spec(A_0)$, $U_0$ can be written as the union of principal opens $U_{0,i} = \Spec(A_0[f_i^{-1}])$ for some elements $f_i \in \pi_0(A_0)$.
      By base change, the opens $U_{0,i} \fibprod_{X_0} X = \Spec(A[f_i^{-1}])$ then cover $X = \Spec(A)$, i.e., we can write $1 \in \pi_0(A)$ as a linear combination of the images of $f_i$.
      But then the same must hold in $\pi_0(A_\alpha)$ for some large enough $\alpha$.
      That is, $U_\alpha = X_\alpha$, so that $\sF_\alpha$ is finite locally free.
    \end{proof}

\section{Algebraic K-theory}
\label{sec:K}

  For a commutative ring $R$, Quillen's algebraic K-theory spectrum $\K(R)$ (see \cite{QuillenK}) can be defined succinctly as the homotopy-theoretic group completion \cite{QuillenGroup,May,GepnerGrothNikolaus} of the groupoid of finitely generated projective $R$-modules, regarded as an $\Einfty$-monoid under direct sum.
  This can be viewed either as an $\Einfty$-group or a connective spectrum.
  This description is also valid for a \scr $R$ (see \cite[Lect.~19, Thm.~5]{LurieKLect}).

  For a scheme $X$, doing the same construction with the groupoid of finite locally free $\sO_X$-modules (or vector bundles over $X$ of finite rank) gives rise to Quillen's original construction of $\K(X)$.
  The Thomason--Trobaugh definition \cite{ThomasonTrobaugh} is more sophisticated: it replaces vector bundles by perfect complexes and requires Waldhausen's $S_\bullet$-construction instead group completion.
  The advantage of Thomason--Trobaugh K-theory is its excellent behaviour over arbitrary (qcqs) schemes: for instance, there are localization and Mayer--Vietoris long exact sequences.
  
  Thomason--Trobaugh K-theory and Quillen K-theory agree in the presence of the resolution property (\ssecref{ssec:res}), which almost always holds for schemes (e.g. any quasi-projective scheme admits the resolution property).
  For stacks however the resolution property essentially amounts to being a global quotient \cite{TotaroRes,Gross}.
  The Thomason--Trobaugh construction becomes therefore even more crucial in the world of stacks.

  \subsection{Definitions}

    Let $\sA$ be a stable \inftyCat.
    The Waldhausen $S_\bullet$-construction makes sense in this setting and can be used to produce the K-theory space of $\sA$.
    Iterating it gives rise to an $\Einfty$-group structure on this space, and hence to a connective spectrum $\K(\sA)$.
    See \cite[Rem.~1.2.2.5]{LurieHA}, \cite[Sect.~7]{BlumbergGepnerTabuada}, and \cite[Sect.~10]{BarwickWald}.
    Using a generalization of the Bass construction (see \cite[Sect.~4]{CisinskiKhan}), one can also produce a Bass--Thomason--Trobaugh K-theory spectrum $\KB(\sA)$ whose connective cover $\KB(\sA)_{\ge 0}$ is $\K(\sA)$.
    Its homotopy groups are denoted
    \[ \K_i(\sA) = \pi_i(\KB(\sA)), \qquad i\in\bZ, \]
    so that $\K_i(\sA) \simeq \pi_i(\K(\sA))$ for $i\ge 0$.
    Any exact functor of stable \inftyCats $\sA \to \sB$ induces maps of spectra $\K(\sA) \to \K(\sB)$ and $\KB(\sA) \to \KB(\sB)$.

    \begin{thm}\label{thm:abstract localization}
      Let $j^* : \sA \to \sB$ be a functor between compactly generated stable \inftyCats with fully faithful right adjoint $j_*$.
      Suppose that $j^*$ is compact (equivalently, preserves compact objects) and its kernel $\sA_0$ is compactly generated.
      Then we have:

      \begin{thmlist}
        \item\emph{Proto-localization.}
        There is a fibre sequence of connective spectra
        \[ \K(\sA_0^\omega) \to \K(\sA^\omega) \to \K(\sB^\omega). \]

        \item\emph{Localization.}
        There is an exact triangle of spectra
        \[ \KB(\sA_0^\omega) \to \KB(\sA^\omega) \to \KB(\sB^\omega). \]
      \end{thmlist}
      In particular, there is a long exact sequence
      \[
        \cdots
        \xrightarrow{\partial} \K_i(\sA_0^\omega)
        \to \K_i(\sA^\omega)
        \to \K_i(\sB^\omega)
        \xrightarrow{\partial} \K_{i-1}(\sA_0^\omega)
        \to \cdots.
      \]
    \end{thm}
    \begin{proof}
      See \cite[Prop.~10.20]{BarwickWald} for the first statement.
      The second follows from the first and \cite[Thm.~4.5.7]{CisinskiKhan}.
    \end{proof}

    \begin{cor}[Additivity]\label{cor:additivity}
      For any semi-orthogonal decomposition of $\sA$ into stable subcategories, the spectra $\K(\sA)$ and $\KB(\sA)$ both split into direct sums.
    \end{cor}
    \begin{proof}
      As in \cite[Lem.~2.8]{KhanKblow}, we may reduce to the case of a semi-orthogonal decomposition of length $2$.
      Applying \thmref{thm:abstract localization} to the ind-completions, we get a fibre sequence (resp. exact triangle), which is split by the map $j_* : \K(\sB) \to \K(\sA)$ (resp. $j_* : \KB(\sB) \to \KB(\sA)$).
    \end{proof}

    \begin{rem}\label{rem:Efimov}
      If the kernel $\sA_0$ is not compactly generated, one can still identify the fibre in \thmref{thm:abstract localization} as the Efimov K-theory of $\sA_0$ (see \cite{HoyoisEfimov}).
    \end{rem}

    \begin{thm}\label{thm:abstract excision}
      Suppose given a commutative square
      \[ \begin{tikzcd}
        \sA \ar{r}{f^*}\ar{d}{p^*}
        & \sB \ar{d}{q^*}
        \\
        \sA' \ar{r}{g^*}
        & \sB'
      \end{tikzcd} \]
      of compactly generated stable \inftyCats and compact colimit-preserving functors.
      If the square is cartesian (in the \inftyCat of large \inftyCats) and $g^*$ has fully faithful right adjoint $g_*$, then the induced square of spectra
      \[ \begin{tikzcd}
        \KB(\sA^\omega) \ar{r}{f^*}\ar{d}{p^*}
        & \KB(\sB^\omega) \ar{d}{q^*}
        \\
        \KB(\sA'^\omega) \ar{r}{g^*}
        & \KB(\sB'^\omega)
      \end{tikzcd} \]
      is cartesian.
    \end{thm}
    \begin{proof}
      The cartesianness of the square implies that $f_*$ is also fully faithful.
      Since the \inftyCat of spectra is stable, it will suffice to show that the horizontal fibres are isomorphic.
      If the kernels of $f^*$ and $g^*$ are compactly generated, this follows immediately from \thmref{thm:abstract localization}.
      By \remref{rem:Efimov}, this holds even otherwise.
      (See \cite[Cor.~13]{HoyoisEfimov} and, for another proof, \cite[Thm.~18]{Tamme}.)
    \end{proof}

    \begin{thm}[Continuity]\label{thm:abstract contin}
      Let $\sA$ be a filtered colimit of stable \inftyCats $\sA_\alpha$.
      Then the canonical maps
      \begin{align*}
        \colim_\alpha \K(\sA_\alpha) \to \K(\sA),\\
        \colim_\alpha \KB(\sA_\alpha) \to \KB(\sA)
      \end{align*}
      are invertible.
    \end{thm}
    \begin{proof}
      For $\K(-)$ this follows by inspection of the Waldhausen $S_\bullet$-construction (see e.g. \cite[Prop.~7.10]{BlumbergGepnerTabuada}).
      For $\KB(-)$, note that it is enough to show the map is invertible after taking the $n$-connective cover $\tau_{\ge n}$ for every $n$.
      The induced map is, by construction of $\KB(-)$ and because $\tau_{\ge n}$ commutes with colimits, a retract of the map
      \[
        \colim_\alpha \K(\sA_\alpha \otimes \sC^{\otimes n}) \to \K(\sA \otimes \sC^{\otimes n})
      \]
      which is invertible by the $\K(-)$ statement.
      Here $\sC$ is the stable \inftyCat of perfect complexes over an analogue of the multiplicative group $\bG_m$ over the sphere spectrum.
      See \cite[Not.~4.2.4]{CisinskiKhan} and compare the proof of \cite[Thm.~4.5.7(ii)]{CisinskiKhan}.
    \end{proof}

    \begin{defn}
      For a \dA stack $\sX$, write
      \[ \K(\sX) = \K(\Dperf(\sX)), \quad \KB(\sX) = \KB(\Dperf(\sX)) \]
      for the Thomason--Trobaugh and Bass--Thomason--Trobaugh K-theory spectra.
    \end{defn}

    From now on we'll mostly restrict our attention to $\KB$, since statements about $\K$ can be recovered simply by passing to connective covers.

  \subsection{Operations}

    The various operations on perfect complexes give rise to operations in K-theory.

    The tensor product on $\Dperf(\sX)$ induces a cup product
    \[ \cup : \KB(\sX) \otimes \KB(\sX) \to \KB(\sX) \]
    which is part of an $\Einfty$-ring structure on $\KB(\sX)$.

    For any morphism $f : \sX \to \sY$, the inverse image functor $f^* : \Dperf(\sY) \to \Dperf(\sX)$ gives rise to a map
    \[ f^* : \KB(\sY) \to \KB(\sX). \]
    Since $f^*$ is symmetric monoidal, this map is multiplicative (a homomorphism of $\Einfty$-ring spectra).

    If $f$ is proper, of finite cohomological dimension, almost of finite presentation, and of finite Tor-amplitude, then there is a Gysin map
    \[ f_* : \KB(\sX) \to \KB(\sY). \]
    By \thmref{thm:uni fin coh dim} this map is $\KB(\sY)$-linear (a homomorphism of $\KB(\sY)$-module spectra).
    In other words:

    \begin{prop}[Projection formula]\label{prop:K proj}
      If $f : \sX \to \sY$ is proper, of finite cohomological dimension, almost of finite presentation, and of finite Tor-amplitude, then we have canonical identifications
      \begin{equation*}
        f_*(x) \cup y
        \simeq f_*(x \cup f^*(y))
      \end{equation*}
      in $\KB(\sY)$, for all $x \in \KB(\sX)$, $y \in \KB(\sY)$.
    \end{prop}

    We also have (by \thmref{thm:uni fin coh dim}):

    \begin{prop}[Base change formula]\label{prop:K base change}
      Suppose given a homotopy cartesian square of \dA stacks
      \begin{equation*}
        \begin{tikzcd}
          \sX' \ar{r}{g}\ar{d}{p}
            & \sY' \ar{d}{q}
          \\
          \sX \ar{r}{f}
            & \sY.
        \end{tikzcd}
      \end{equation*}
      If $f$ is proper, of finite cohomological dimension, almost of finite presentation, and of finite Tor-amplitude, then we have a canonical homotopy
      \begin{equation*}
        q^*f_* \simeq g_*p^*
      \end{equation*}
      of maps $\KB(\sX) \to \KB(\sY')$.
    \end{prop}

    If $Z \sub \abs{\sX}$ is a cocompact closed subset, then we write
    $$\KB(\sX~\mrm{on}~Z) := \KB(\Dperf(\sX~\mrm{on}~Z))$$
    for the Bass--Thomason--Trobaugh K-theory spectrum of $\Dperf(\sX~\mrm{on}~Z)$, defined as in \defnref{defn:0pg10u70g}.
    
    \begin{thm}[Localization]\label{thm:K localization}
      If $\sX$ is absolutely perfect, then for every cocompact closed subset $Z\sub\abs{\sX}$ there is a canonical exact triangle of spectra
      \[ \KB(\sX~\mrm{on}~Z) \to \KB(\sX) \xrightarrow{j^*} \KB(\sX\setminus Z). \]
    \end{thm}
    \begin{proof}
      Since $\sX$ is perfect, so is $\sX\setminus Z$ (since $j^* : \Dqc(\sX) \to \Dqc(\sX\setminus Z)$ is compact and generates under colimits).
      Hence we have $\Dperf(\sX) \simeq \Dqc(\sX)^\omega$ and similarly for $\sX\setminus Z$, as well as $\Dperf(\sX~\mrm{on}~Z) \simeq \Dqc(\sX~\mrm{on}~Z)^\omega$.
      Thus we may apply \thmref{thm:abstract localization}.
    \end{proof}

  \subsection{Projective bundles and blow-ups}

    Let $\sX$ be a \dA stack.
    For any locally free sheaf $\sE$ on $\sX$ of rank $n+1$, $n\ge 0$, we may consider the associated projective bundle
    $$q : \P_\sX(\sE) \to \sX.$$
    This is the universal derived stack over $\sX$ equipped with a surjection $q^*(\sE) \twoheadrightarrow \sO(1)$ onto an invertible sheaf.
    The morphism $q$ is smooth of relative dimension $n$, proper and schematic (see \cite[Prop.~3.1]{KhanKblow}).

    The following is a generalization of \cite[Thm.~7.3]{ThomasonTrobaugh}, \cite{ThomasonProj}.

    \begin{thm}\label{thm:K projbun}
      The maps
      \[
        \KB(X) \to \KB(\P_\sX(\sE)),
        \quad x \mapsto q^*(x) \cup [\sO(-k)],
      \]
      induce an isomorphism
      \[
        \KB(\P_\sX(\sE)) \simeq \bigoplus_{k=0}^{n} \KB(\sX).
      \]
    \end{thm}

    \begin{proof}
      The standard semi-orthogonal decomposition of Orlov--Thomason on $\Dperf(\P_\sX(\sE))$ extends to this setting by \cite[Thm.~B]{KhanKblow}.
      Thus the formula follows from \corref{cor:additivity}.
      See \cite[Cor.~3.6]{KhanKblow}.
    \end{proof}

    For any quasi-smooth closed immersion of \dA stacks $i : \sZ \to \sX$, say of virtual codimension $n$, we may form the derived blow-up $\Bl_\sZ\sX$ as in \cite{KhanRydh}.
    This fits in a commutative square
    \begin{equation*}
      \begin{tikzcd}
        \P_\sZ(\sN_{\sZ/\sX}) \ar{r}{i_D}\ar{d}{q}
        & \Bl_\sZ\sX\ar{d}{p}
        \\
        \sZ\ar{r}{i}
        & \sX
      \end{tikzcd}
    \end{equation*}
    which is cartesian on classical truncations (but not homotopy cartesian unless $n=1$).
    The upper closed immersion $i_D$ is a virtual Cartier divisor, i.e., a quasi-smooth closed immersion of virtual codimension $1$.
    The morphism $q : \bP_\sZ(\sN_{\sZ/\sX}) \to \sZ$ is the projection of the projective bundle associated to the conormal sheaf $\sN_{\sZ/\sX}$.
    The projection $p : \Bl_\sZ\sX \to \sX$ is quasi-smooth of relative virtual dimension $0$, proper, and schematic.
    The following is a generalization of \cite{ThomasonBlow}.

    \begin{thm}\label{thm:K blow}
      The maps $p^* : \KB(\sX) \to \KB(\Bl_\sZ\sX)$ and 
      \[
        \KB(\sZ) \to \KB(\Bl_\sZ\sX),
        \quad x\mapsto i_{D,*}\left(q^*(x) \cup [\sO(-k)]\right),
      \]
      induce an isomorphism
      \[
        \KB(\Bl_\sZ\sX) \simeq \KB(\sX) \oplus \bigoplus_{k=1}^{n-1} \KB(\sZ).
      \]
    \end{thm}

    \begin{proof}
      The standard semi-orthogonal decomposition of Orlov--Thomason on $\Dperf(\Bl_\sZ\sX)$ extends to this setting by \cite[Thm.~C]{KhanKblow}.
      Thus the formula follows from \corref{cor:additivity}.
      See \cite[Cor.~4.4]{KhanKblow}.
    \end{proof}

    \begin{cor}
      The induced square
      \[ \begin{tikzcd}
        \KB(\sX) \ar{r}{i^*}\ar{d}{p^*}
        & \KB(\sZ) \ar{d}{q^*}
        \\
        \KB(\Bl_\sZ\sX) \ar{r}{i_D^*}
        & \KB(\P_\sZ(\sN_{\sZ/\sX}))
      \end{tikzcd} \]
      is cartesian.
    \end{cor}
    \begin{proof}
      Combine Theorems~\ref{thm:K projbun} and \ref{thm:K blow}.
    \end{proof}

  \subsection{Excising closed subsets}

    \begin{thm}\label{thm:K excision}
      Let $\sX$ and $\sX'$ be perfect \dA stacks.
      Let $f : \sX' \to \sX$ be a morphism and $Z \sub \abs{\sX}$ a cocompact closed subset.
      Suppose one of the following conditions holds:
      \begin{thmlist}
        \item\emph{Formal neighbourhood}.
         The morphism $f$ is representable and is an isomorphism infinitely near $Z$.
         That is, the induced morphism $\sX'^\wedge_{f^{-1}(Z)} \to \sX^\wedge_{Z}$ on formal completions is invertible.

        \item\emph{Étale neighbourhood}.\label{item:K excision etale}
        The morphism $f$ is étale and restricts to an isomorphism $f^{-1}(Z)_\red \to Z_\red$.
      \end{thmlist}
      Then the induced square
      \[ \begin{tikzcd}
        \KB(\sX) \ar{r}\ar{d}{f^*}
          & \KB(\sX\setminus Z) \ar{d}
        \\
        \KB(\sX') \ar{r}
          & \KB(\sX'\setminus f^{-1}(Z))
      \end{tikzcd} \]
      is cartesian.
    \end{thm}
    \begin{proof}
      Follows from \thmref{thm:abstract excision} and the corresponding statement at the level of $\Dqc$, see e.g. \cite[Thm.~4.1.1, Rem.~4.1.4]{BachmannKhanRaviSosnilo}.
    \end{proof}

    \begin{cor}\label{cor:K Zariski excision}
      Let $\sX$ be a perfect \dA stack.
      If $\sU \sub \sX$ and $\sV \sub \sX$ are open substacks such that $\sX = \sU \cup \sV$, then the square
      \[ \begin{tikzcd}
        \KB(\sX) \ar{r}\ar{d}
          & \KB(\sU) \ar{d}
        \\
        \KB(\sV) \ar{r}
          & \KB(\sU\cap\sV)
      \end{tikzcd} \]
      is cartesian.
      In particular, there is a long exact Mayer--Vietoris sequence
      \[
        \cdots
        \xrightarrow{\partial} \K_i(\sX)
        \to \K_i(\sU) \oplus \K_i(\sV)
        \to \K_i(\sU\cap\sV)
        \xrightarrow{\partial} \K_{i-1}(\sX) \to \cdots.
      \]
    \end{cor}
    \begin{proof}
      Apply \thmref{thm:K excision} with $\sZ \to \sX$ a closed immersion complementary to $\sU$ and $\sX' = \sV$.
    \end{proof}

    \begin{rem}\label{rem:K descent}
      By \cite[Thm.~7.3.5.2]{LurieHTT}, \corref{cor:K Zariski excision} implies that $\KB$ satisfies Zariski descent on the site of perfect \dA stacks.
      Similarly \thmref{thm:K excision}\itemref{item:K excision etale} implies that it satisfies descent for Nisnevich's completely decomposed étale topology.
      Indeed this topology is generated by families $\{f : \sX' \to \sX, j : \sX \setminus Z \hook \sX\}$, where $f$ is an étale neighbourhood of a cocompact closed $Z\sub \abs{\sX}$ and $j$ is the inclusion of the complement (see \cite[Prop.~2.9]{HoyoisKrishna} for the case of stacks, \cite[\S 3, Prop.~1.4]{MorelVoevodsky} for the case of noetherian finite-dimensional schemes).
      Hence the claim follows from \cite[Thm.~2.2.7]{KhanLocalization}.
    \end{rem}

  \subsection{Excising open subsets}
  \label{ssec:excising opens}

    The problem of excising open subsets is more subtle.
    We are given a morphism of \dA stacks $f : \sX' \to \sX$ and quasi-compact open subsets $\sU \sub \sX$, $\sU' \sub \sX'$.
    Let $\sZ \to \sX$ and $\sZ' \to \sX'$ be closed immersions complementary to $\sU$ and $\sU'$, respectively, fitting in a commutative square
    \[ \begin{tikzcd}
      \sZ' \ar{r}\ar{d}
      & \sX' \ar{d}{f}
      \\
      \sZ \ar{r}
      & \sX
    \end{tikzcd} \]
    which is cartesian on underlying classical truncations.

    For example, if $f$ is affine and the square is also cocartesian, then it is called a \emph{Milnor square}.
    Algebraic K-theory does not generally satisfy excision for Milnor squares (see \cite{Swan}), but it does for Milnor squares that are \emph{homotopy} cartesian:

    \begin{thm}\label{thm:K Milnor}
      Suppose given a Milnor square as above consisting of \emph{perfect} \dA stacks.
      If the square is moreover homotopy cartesian, then the induced square of spectra
      \[ \begin{tikzcd}
        \KB(\sX) \ar{r}\ar{d}
        & \KB(\sZ) \ar{d}
        \\
        \KB(\sX') \ar{r}
        & \KB(\sZ')
      \end{tikzcd} \]
      is cartesian.
    \end{thm}
    \begin{proof}
      See \cite[Thm.~A]{LandTamme} in the affine case and \cite{BachmannKhanRaviSosnilo} for the case of stacks.
    \end{proof}

    For a general Milnor square, we can still get an excision statement if we replace the closed substacks by their formal completions.
    In fact, we have the following refinement where we consider the natural pro-spectra $\form{\KB}(\sX^\wedge_\sZ)$, $\form{\KB}(\sX'^\wedge_{\sZ'})$ associated to the formal completions (viewed as ind-stacks).

    \begin{thm}\label{thm:K Milnor pro}
      For any Milnor square as above consisting of noetherian derived $1$-Artin stacks with bounded structure sheaves, affine diagonals, and nice stabilizers, the induced square of pro-spectra
      \[ \begin{tikzcd}
        \{\KB(\sX)\}\ar{r}\ar{d}
          & \form{\KB}(\sX^\wedge_{\sZ})\ar{d}
        \\
        \{\KB(\sX')\}\ar{r}
          & \form{\KB}(\sX'^\wedge_{\sZ'})
      \end{tikzcd} \]
      is cartesian.
    \end{thm}
    \begin{proof}
      See \cite{BachmannKhanRaviSosnilo}.
      In the affine case see also \cite[Thm.~2.32]{LandTamme} where it is proven that the square is ``weakly cartesian''.
    \end{proof}

    Another interesting case is that of a proper representable morphism $f : \sX' \to \sX$ which restricts to an isomorphism $\sU' \to \sU$.
    Then the square is often called a \emph{proper cdh square} or \emph{abstract blow-up square}.
    K-theory typically doesn't satisfy excision with respect to such squares either, but again it holds if we pass to formal completions.

    \begin{thm}[Proper excision]\label{thm:K proper excision}
      Let $f : \sX' \to \sX$ be a proper representable morphism of noetherian derived $1$-Artin stacks with bounded structure sheaves, affine diagonal, and nice stabilizers.
      Let $\sU \sub \sX$ and $\sU' \sub \sX'$ be quasi-compact open subsets over which $f$ restricts to an isomorphism $\sU' \simeq \sU$.
      If $Z$ and $Z'$ are the respective set-theoretic complements of $\sU$ and $\sU'$, then the induced square of pro-spectra
      \[ \begin{tikzcd}
        \{\KB(\sX)\}\ar{r}\ar{d}
          & \form{\KB}(\sX^\wedge_{Z})\ar{d}
        \\
        \{\KB(\sX')\}\ar{r}
          & \form{\KB}(\sX'^\wedge_{Z'})
      \end{tikzcd} \]
      is cartesian.
    \end{thm}

    \begin{rem}
      For \emph{regular} (nonsingular) schemes and stacks, we will see in \ssecref{ssec:KH} that the analogues of Theorems~\ref{thm:K Milnor pro} and \ref{thm:K proper excision} do hold without passing to formal completions; see Theorems~\ref{thm:KH Milnor} and \ref{thm:KH proper excision}.
    \end{rem}

    \begin{rem}\label{rem:affine low}
      Alternatively, if we restrict our attention to affine (but possibly singular) derived schemes and only look at ``very low'' K-groups, then again Milnor excision and proper excision hold before formal completion; see \cite[Thm~3.3]{Milnor}, \cite[Chap.~XII, Thm.~8.3]{Bass}.
      We will come back to this point in \ssecref{ssec:Weibel}.
    \end{rem}

  \subsection{Bass fundamental sequence}

    The Bass fundamental sequence is one of the few ways to understand negative K-groups.
    In fact, the construction of Bass--Thomason--Trobaugh K-theory $\KB$ is rigged to make it hold for the negative K-groups.

    \begin{thm}\label{thm:Bass}
      Let $\sX$ be a perfect \dA stack.
      Then for every integer $n$ there is an exact sequence of abelian groups
      \[
        0
        \to \K_n(\sX)
        \to \K_n(\sX \times \A^1) \oplus \K_n(\sX \times \A^1)
        \to \K_n(\sX \times \bG_m)
        \xrightarrow{\partial} \K_{n-1}(\sX)
        \to 0,
      \]
      functorial in $\sX$ with respect to inverse images.
      Moreover, the map $\partial$ admits a natural $\K_*(\sX)$-module splitting.
    \end{thm}
    \begin{proof}
      By \corref{cor:K Zariski excision} there is a Mayer--Vietoris sequence for the standard affine cover of $\sX \times \P^1$:
      \begin{multline*}
        \cdots
        \to \K_{n+1}(\sX \times \bG_m)
        \xrightarrow{\partial} \K_n(\sX \times \P^1)
        \to \K_n(\sX \times \A^1) \oplus \K_n(\sX \times \A^1)
        \\
        \to \K_n(\sX \times \bG_m)
        \xrightarrow{\partial} \cdots
      \end{multline*}
      Now apply the projective bundle formula (\thmref{thm:K projbun}).
      The splitting comes from the Bott class $b \in \K(\sX \times \bG_m)[-1]$.
      See \cite[Thm.~7.5]{ThomasonTrobaugh} or \cite[Thm.~4.3.1, Rem.~4.3.2]{CisinskiKhan} for details.
    \end{proof}

  \subsection{Continuity}

    \begin{thm}\label{thm:contin K}
      Let $\sX$ be a perfect \dA stack.
      If $\sX$ is the limit of a cofiltered system $(\sX_\alpha)_\alpha$ of perfect \dA stacks with affine transition morphisms, then the canonical map
      \[
        \colim_\alpha \KB(\sX_\alpha)
        \to \KB(\sX)
      \]
      is invertible.
      In particular, there are canonical isomorphisms
      \[
        \colim_\alpha \K_n(X_\alpha)
        \simeq \K_n(X)
      \]
      for all integers $n\in\bZ$.
    \end{thm}
    \begin{proof}
      Combine \thmref{thm:contin}\itemref{item:contin/perf} and \thmref{thm:abstract contin}.
    \end{proof}

  \subsection{Nil-invariance}
  \label{ssec:K nil}

    The following result shows that, up to $\K_{n+1}$, the K-groups of a \scr $R$ are only sensitive to the first $n$ homotopy groups of $R$.

    \begin{prop}\label{prop:K low inv}
      Let $R$ be a \scr.
      For every integer $n$, consider the canonical homomorphism $R \to \tau_{\le n}(R)$ to the $n$-truncation (set $\tau_{\le n}(R) := \tau_{\le0}(R) = \pi_0(R)$ for $n<0$).
      Then the inverse image map
      \[ \KB(\Spec(R)) \to \KB(\Spec(\tau_{\le k}R)) \]
      induces an isomorphism $\K_{i}(\Spec(R)) \simeq \K_{i}(\Spec(\tau_{\le k}R))$ for all $i\le k+1$.
    \end{prop}
    \begin{proof}
      See \cite[Prop.~4.2]{KhanKdescent} or \cite[Prop.~5.1.3]{BachmannKhanRaviSosnilo}.
    \end{proof}

    In particular, an affine derived scheme $X$ and its classical truncation $X_\cl$ have the same K-groups up to $\K_1$.
  
    \begin{prop}[Nil-invariance]\label{prop:K nil invariance}
      Let $i : Z \to X$ be a surjective closed immersion of affine derived schemes.
      Then the map $i^* : \KB(X) \to \KB(Z)$ is $1$-connective; i.e., it is surjective on $\K_1$ and bijective on $\K_n$ for all $n\le 0$.
    \end{prop}
    \begin{proof}
      By \propref{prop:K low inv} we may replace $X$ and $Z$ by their classical truncations.
      By the Bass fundamental sequence (\thmref{thm:Bass}) it suffices to show surjectivity on $\K_1$ and bijectivity on $\K_0$.
      This is \cite[Chap.~IX, Prop.~1.3]{Bass}.
    \end{proof}

    \begin{rem}\label{rem:K nil fail}
      \propref{prop:K nil invariance} does not generalize to higher K-groups.
      Contrast with \thmref{thm:KH derived invariance}.
    \end{rem}

    \begin{rem}
      For non-affine derived schemes or algebraic spaces, \propref{prop:K nil invariance} still holds below the Krull dimension $d$ (i.e., any surjective closed immersion will induce a bijection on $\K_n$ for $n<-d$).
      For \dA stacks (with affine diagonal and nice stabilizers), the same holds if we replace $d$ by the Nisnevich cohomological dimension (see \cite[Cor.~5.1.4]{BachmannKhanRaviSosnilo}).
    \end{rem}

\section{G-theory}
\label{sec:G}

  \subsection{Definition and basic properties}

    For a noetherian \dA stack $\sX$, the \emph{G-theory spectrum} $\G(\sX)$ is defined as the algebraic K-theory of the stable \inftyCat of coherent complexes on $\sX$:
    \[ \G(\sX) = \K(\Dcoh(\sX)). \]

    \begin{prop}\label{prop:G connective}
      Let $\sX$ be a noetherian \dA stack.
      Then the canonical map
      \[ \G(\sX) = \K(\Dcoh(\sX)) \to \KB(\Dcoh(\sX)) \]
      is invertible.
      In other words, the spectrum $\KB(\Dcoh(\sX))$ is connective.
    \end{prop}
    \begin{proof}
      Since $\Dcoh(\sX)$ admits a bounded t-structure with noetherian heart, this follows from \cite[Thm.~1.2]{AntieauGepnerHeller}.
    \end{proof}

    \begin{rem}
      Most of our discussion on G-theory can be extended from noetherian to qcqs stacks without modification.
      \propref{prop:G connective} seems to be an exception: I don't know whether $\KB(\Dcoh(\sX))$ will be connective for $\sX$ non-noetherian (cf. \cite[Conj.~B]{AntieauGepnerHeller}).
    \end{rem}

    \begin{prop}\label{prop:G Quillen}
      Let $\sX$ be a noetherian \dA stack.
      Denote by $\Coh(\sX)$ the abelian category of coherent sheaves on $\sX$ and by $\K(\Coh(\sX))$ its K-theory in the sense of Quillen.
      Then there is a canonical isomorphism
      \[ \K(\Coh(\sX)) \to \G(\sX), \]
      functorial in $\sX$.
    \end{prop}
    \begin{proof}
      Since $\Dcoh(\sX)$ admits a bounded t-structure with heart $\Coh(\sX)$, this follows from \cite[Thm.~6.1]{BarwickHeart}.
    \end{proof}

    \begin{cor}[Derived invariance]\label{cor:G derived invariance}
      Let $\sX$ be a noetherian \dA stack and write $i : \sX_\cl \to \sX$ for the inclusion of the classical truncation.
      Then the direct image map
      \[ i_* : \G(\sX_\cl) \to \G(\sX) \]
      is invertible.
    \end{cor}
    \begin{proof}
      Follows immediately from \propref{prop:G Quillen}, since $i_* : \Dcoh(\sX_\cl) \to \Dcoh(\sX)$ is t-exact and induces an equivalence $\Coh(\sX_\cl) \simeq \Coh(\sX)$ on hearts.
    \end{proof}

    \begin{thm}[Poincaré duality]\label{thm:Poincare}
      Let $\sX$ be a noetherian \dA stack.
      If $\sX$ has bounded structure sheaf, then the inclusion $\Dperf(\sX) \sub \Dcoh(\sX)$ induces canonical maps
      \[ \K(\sX) \to \KB(\sX) \to \KB(\Dcoh(\sX)) \simeq \G(\sX) \]
      which are invertible if $\sX$ is regular.
    \end{thm}
    \begin{proof}
      If $\sX$ is regular, then the inclusion $\Dperf(\sX) \sub \Dcoh(\sX)$ is an equality, hence the second map is invertible.
      From this it follows that $\KB(\sX)$ is connective, so the first map is also invertible (since it is a connective cover).
    \end{proof}

    \begin{rem}
      The map $\K(\sX) \to \G(\sX)$, or its factorization through $\KB(\sX)$, is sometimes called the \emph{Cartan map}.
    \end{rem}

  \subsection{Operations}

    Since tensoring with a perfect complex preserves coherence, we get a cap product
    \[ \cap : \KB(\sX) \otimes \G(\sX) \to \G(\sX) \]
    which is part of a $\KB(\sX)$-module structure on $\G(\sX)$.

    For any morphism $f : \sX \to \sY$ of finite Tor-amplitude, the inverse image functor $f^* : \Dcoh(\sY) \to \Dcoh(\sX)$ gives rise to a Gysin map
    \[ f^* : \G(\sY) \to \G(\sX). \]
    This is compatible with the K-theoretic inverse image under the maps in \thmref{thm:Poincare}.

    Let $f$ be a proper morphism.
    If $f$ is representable, or more generally of finite cohomological dimension, then there is a direct image map
    \[ f_* : \G(\sX) \to \G(\sY). \]
    This is compatible with the K-theoretic direct image (when $f$ is of finite Tor-amplitude and almost of finite presentation) under the maps in \thmref{thm:Poincare}.
    By \thmref{thm:uni fin coh dim} it is also $\KB(\sY)$-linear (a homomorphism of $\KB(\sY)$-module spectra).
    In other words:

    \begin{prop}[Projection formula]\label{prop:G projection}
      If $f : \sX \to \sY$ is a proper and of finite cohomological dimension morphism of noetherian \dA stacks, then we have canonical identifications
      \begin{equation*}
        y \cap f_*(x)
        \simeq f_*(f^*(y) \cap x)
      \end{equation*}
      in $\G(\sY)$, for all $x \in \G(\sX)$, $y \in \KB(\sY)$.
    \end{prop}

    \thmref{thm:uni fin coh dim} similarly implies:

    \begin{prop}[Base change formula]\label{prop:G base change}
      Suppose given a homotopy cartesian square of noetherian \dA stacks
      \begin{equation*}
        \begin{tikzcd}
          \sX' \ar{r}{g}\ar{d}{p}
            & \sY' \ar{d}{q}
          \\
          \sX \ar{r}{f}
            & \sY.
        \end{tikzcd}
      \end{equation*}
      If $f$ is proper and of finite cohomological dimension and $q$ is of finite Tor-amplitude, then we have a canonical homotopy
      \begin{equation*}
        q^*f_* \simeq g_*p^*
      \end{equation*}
      of maps $\G(\sX) \to \G(\sY')$.
    \end{prop}

    For G-theory we have the following analogue of \thmref{thm:K localization}:

    \begin{thm}[Localization]\label{thm:G localization}
      Let $i : \sZ \to \sX$ be a closed immersion of noetherian \dA stacks with complementary open immersion $j : \sU \to \sX$.
      Then there is a canonical exact triangle of spectra
      \begin{align*}
        \G(\sZ)
        \xrightarrow{i_*} \G(\sX)
        \xrightarrow{j^*} \G(\sU).
      \end{align*}
    \end{thm}
    \begin{proof}
      By \propref{prop:G Quillen}, this follows from Quillen's dévissage and localization theorems \cite[Sect.~5, Thms.~4 and 5]{QuillenK}.
    \end{proof}

    \begin{cor}[Nil invariance]\label{cor:G nil-invariance}
      Let $i : Z \to X$ be a surjective closed immersion of noetherian \dA stacks.
      Then the induced morphism
      \[ i_* : \G(\sZ) \to \G(\sX) \]
      is invertible.
    \end{cor}

  \subsection{Excision}

    \begin{thm}[Étale excision]\label{cor:G Nis}
      Let $\sX$ be a noetherian \dA stack.
      Let $j : \sU \to \sX$ be an open immersion, and $Z \sub \abs{\sX}$ the set-theoretic complement.
      Then for any étale neighbourhood $f : \sX' \to \sX$ of $Z$ (i.e., $f$ is an étale morphism which restricts to an isomorphism $f^{-1}(Z)_\red \to Z_\red$), the induced square
      \[ \begin{tikzcd}
        \G(\sX) \ar{r}\ar{d}
          & \G(\sU) \ar{d}
        \\
        \G(\sX') \ar{r}
          & \G(f^{-1}(\sU))
      \end{tikzcd} \]
      is cartesian.
    \end{thm}
    \begin{proof}
      By \thmref{thm:G localization}, the horizontal fibres are isomorphic.
    \end{proof}

    \begin{thm}[Proper co-excision]\label{thm:proper co-excision G}
      Let $\sX$ be a noetherian \dA stack.
      Let $i : \sZ \to \sX$ be a closed immersion with open complement $j: \sU \to \sX$.
      Then for any proper morphism $f : \sX' \to \sX$ which restricts to an isomorphism $f^{-1}(\sU)_\red \to \sU_\red$, the induced square
      \[ \begin{tikzcd}
        \G(f^{-1}(\sZ)) \ar{r}\ar{d}
        & \G(\sX') \ar{d}{f_*}
        \\
        \G(\sZ) \ar{r}{i_*}
        & \G(\sX)
      \end{tikzcd} \]
      is cocartesian (and hence cartesian).
    \end{thm}
    \begin{proof}
      By \thmref{thm:G localization}, the horizontal cofibres are isomorphic.
    \end{proof}

  \subsection{Projective bundles and blow-ups}

    We have the following G-theory analogues of Theorems~\ref{thm:K projbun} and \ref{thm:K blow}.

    \begin{thm}\label{thm:G projbun}
      Let $\sX$ be a noetherian \dA stack, $\sE$ a locally free sheaf on $\sX$ of rank $n+1$, $n\ge 0$, and $q : \P_\sX(\sE) \to X$ the associated projective bundle.
      Then the maps
      \[
        \G(X) \to \G(\P_\sX(\sE)),
        \quad x \mapsto q^*(x) \cup [\sO(-k)]
      \]
      induce an isomorphism
      \[
        \G(\P_\sX(\sE)) \simeq \bigoplus_{k=0}^{n} \G(\sX).
      \]
    \end{thm}

    \begin{proof}
      The standard semi-orthogonal decomposition on $\Dqc(\P_\sX(\sE))$ (see \cite[Thm.~3.3]{KhanKblow}) restricts to coherent complexes, so the formula follows from \corref{cor:additivity}.
    \end{proof}

    \begin{thm}\label{thm:G blow}
      Let $\sX$ be a noetherian \dA stack.
      For any quasi-smooth closed immersion $i : \sZ \to \sX$ of virtual codimension $n$, consider the blow-up square
      \begin{equation}\label{eq:aosufnq}
        \begin{tikzcd}
          \P_\sZ(\sN_{\sZ/\sX}) \ar{r}{i_D}\ar{d}{q}
          & \Bl_\sZ\sX\ar{d}{p}
          \\
          \sZ\ar{r}{i}
          & \sX
        \end{tikzcd}
      \end{equation}
      Then the maps $p^* : \G(\sX) \to \G(\Bl_\sZ\sX)$ and 
      \[
        \G(\sZ) \to \G(\Bl_\sZ\sX),
        \quad x\mapsto i_{D,*}\left(q^*(x) \cup [\sO(-k)]\right),
      \]
      induce an isomorphism
      \[
        \G(\Bl_\sZ\sX) \simeq \G(\sX) \oplus \bigoplus_{k=1}^{n-1} \G(\sZ).
      \]
    \end{thm}

    \begin{proof}
      The standard semi-orthogonal decomposition on $\Dqc(\Bl_\sZ\sX)$ (see \cite[Thm.~4.3]{KhanKblow}) restricts to coherent complexes.
      Thus the formula follows from \corref{cor:additivity}.
    \end{proof}

    \begin{cor}
      Let $\sX$ be a noetherian \dA stack.
      For any quasi-smooth closed immersion $i : \sZ \to \sX$, the blow-up square \eqref{eq:aosufnq} induces a cartesian square
      \[ \begin{tikzcd}
        \G(\sX) \ar{r}{i^*}\ar{d}{p^*}
        & \G(\sZ) \ar{d}{q^*}
        \\
        \G(\Bl_\sZ\sX) \ar{r}{i_D^*}
        & \G(\P_\sZ(\sN_{\sZ/\sX})).
      \end{tikzcd} \]
    \end{cor}
    \begin{proof}
      Combine Theorems~\ref{thm:G projbun} and \ref{thm:G blow}.
    \end{proof}

  \subsection{Homotopy invariance}

    Let $\sX$ be a noetherian \dA stack.
    Let $\sE$ be a locally free sheaf of finite rank, $\phi : \sE \twoheadrightarrow \sO_\sX$ a surjective\footnote{%
      on $\pi_0$
    } $\sO_\sX$-module homomorphism, and $\sF$ its fibre (which is locally free of finite rank).
    The associated \emph{affine bundle}
    \[ \pi : \bV_\sX(\sE,\phi) \to \sX \]
    is the moduli of splittings of $\phi$.
    It fits in a closed/open pair
    \[ \P_\sX(\sF^\vee) \hook \P_\sX(\sE^\vee) \hookleftarrow \bV_\sX(\sE,\phi) \]
    and it is a torsor under the vector bundle $\bV_\sX(\sF^\vee) \to \sX$.
    Every vector bundle torsor arises via this construction, see e.g. the proof of \cite[Thm.~4.1]{ThomasonEquiv}.

    \begin{thm}\label{thm:G htp}
      Let $\sX$ be a noetherian \dA stack.
      For any $(\sE, \phi)$ as above, let $\pi : \bV_\sX(\sE,\phi) \to \sX$ denote the associated affine bundle.
      Then the inverse image map
      \begin{equation*}
        \pi^* : \G(\sX) \to \G(\bV_\sX(\sE,\phi))
      \end{equation*}
      is invertible.
    \end{thm}

    \begin{proof}
      Note that $\pi$ is smooth, so the map exists.
      Combine the localization triangle (\thmref{thm:G localization})
      \begin{equation*}
        \G(\P_\sX(\sF^\vee)) \to \G(\P_\sX(\sE^\vee)) \to \G(\bV_\sX(\sE,\phi))
      \end{equation*}
      with the projective bundle formula (\thmref{thm:G projbun}).
    \end{proof}

    For any finite locally free sheaf $\sE$, the affine bundle associated to the projection $\sE \oplus \sO_\sX \twoheadrightarrow \sO_\sX$ is the vector bundle $\bV_\sX(\sE^\vee)$ classifying sections $\sO_\sX \to \sE$.
    Therefore we in particular get homotopy invariance for vector bundles.

    \begin{cor}\label{cor:G htp vb}
      Let $\sX$ be a noetherian \dA stack.
      For any finite locally free sheaf $\sE$ on $\sX$, let $\pi : \bV_\sX(\sE) \to \sX$ denote the vector bundle parametrizing cosections of $\sE$.
      Then the inverse image map
      \begin{equation*}
        \pi^* : \G(\sX) \to \G(\bV_\sX(\sE))
      \end{equation*}
      is invertible.
    \end{cor}

    We can also generalize this further to vector bundle stacks.

    \begin{cor}\label{cor:G htp vb stack}
      Let $\sX$ be a noetherian \dA stack.
      Assume that $\sX$ has affine stabilizers.
      For any perfect complex $\sE$ on $\sX$ of Tor-amplitude $\le 0$, let $\pi : \bV_\sX(\sE) \to \sX$ denote the vector bundle stack parametrizing cosections of $\sE$.
      Then the inverse image map
      \[ \pi^* : \G(\sX) \to \G(\bV_\sX(\sE)) \]
      is invertible.
    \end{cor}

    \begin{proof}
      If $\sE$ is of Tor-amplitude $[0,0]$, then it is finite locally free and we are in the case of \corref{cor:G htp vb}.

      If $\sE$ is of Tor-amplitude $[-1,-1]$, then $E := \bV_\sX(\sE)$ is the classifying stack of the vector bundle $E[-1] := \bV_\sX(\sE[1])$.
      In this case, the canonical morphism $\sigma : X \twoheadrightarrow [X/E[-1]] \simeq E$ is the projection of a vector bundle (it is the quotient by $E[-1]$ of the projection $E[-1] \to X$).
      Thus $\sigma^*$ is invertible by \corref{cor:G htp vb}.
      Since $\sigma^*\pi^* = \id$, we deduce that $\pi^*$ is invertible.
      Repeating the same argument inductively also gives the case of Tor-amplitude $[-k,-k]$, for any $k\ge 0$.

      Finally let $\sE$ be of Tor-amplitude $[-k,0]$ for some $k>0$.
      Since $\sX$ has affine stabilizers, its classical truncation admits a stratification by stacks with the resolution property (see \cite[Prop.~2.6(i)]{HallRydhGroups}).
      By \corref{cor:G derived invariance} and \thmref{thm:G localization}, we may therefore assume that $\sX$ admits the resolution property.
      In this case we can find a morphism $\phi : \sE_0[-k] \to \sE$ with $\sE_0$ finite locally free, whose cofibre $\sE'$ is of Tor-amplitude $[-k+1,0]$.
      Write $E := \bV_\sX(\sE)$, $E' := \bV_\sX(\sE')$, and $E_0[k] := \bV_\sX(\sE_0[-k])$.
      The projection $\pi : E \to \sX$ factors through the morphism $\phi : E \to E_0[k]$ and the projection $\pi_0 : E_0[k] \to \sX$.
      Since $\phi$ is the projection of a vector bundle (it is the quotient by $E_0[k-1]$ of the projection $\pi' : E' \to \sX$), $\phi^*$ is invertible by the Tor-amplitude $[0,0]$ case.
      By the $[-k,-k]$ case above, $\pi_0^*$ is also invertible.
      Hence $\pi^* \simeq \phi^*\pi_0^*$ is invertible.
    \end{proof}

\section{Homotopy invariance and singularities}
\label{sec:KH}

  \subsection{Weibel's conjecture}
  \label{ssec:Weibel}

    Let $X$ be a noetherian algebraic space.
    If $X$ is regular, then combining Poincaré duality (\thmref{thm:Poincare}) with \propref{prop:G connective} and \corref{cor:G htp vb} yields:
    \begin{defnlist}
      \item 
      The negative K-groups $\K_{-n}(X)$ vanish for all $n>0$.

      \item 
      For any vector bundle $\pi : E \to X$, the map $\pi^* : \KB(X) \to \KB(E)$ is invertible.
    \end{defnlist}

    When $X$ is singular, both these properties fail.
    In fact, as we will see in the next subsection, all the ``pathological'' behaviour of algebraic K-theory for singular schemes, such as failure of nil-invariance and excision of open subsets (see Subsects.~\ref{ssec:K nil} and \ref{ssec:excising opens}) can be traced back to the failure of homotopy invariance.
    In particular, these excision statements will turn out to be true for nonsingular spaces.

    On the other hand, many of these pathologies disappear on ``low enough'' K-groups (see \propref{prop:K nil invariance} and \remref{rem:affine low}).
    This is in some sense ``explained'' by the following statement, which says that even for singular spaces, homotopy invariance does hold on low enough K-groups.

    \begin{thm}\label{thm:Weibel}
      Let $X$ be a noetherian algebraic space.
      Suppose that $X$ is pro-smooth over a noetherian algebraic space of Krull dimension $d$.
      Then we have:
      \begin{thmlist}
        \item\label{item:Weibel/van}
        The negative K-groups $\K_{-n}(X)$ vanish for all $n>d$.

        \item\label{item:Weibel/htp}
        For any vector bundle $\pi : E \to X$, the maps $\pi^* : \K_{-n}(X) \to \K_{-n}(E)$ are invertible for all $n\ge d$.
      \end{thmlist}
    \end{thm}

    \begin{exam}
      If $X$ is smooth (or pro-smooth) over a field, then it is regular and \thmref{thm:Weibel} just recovers the nonsingular situation discussed above.
      (Conversely, if we restrict our attention to affine schemes over a field, then by Popescu's desingularization theorem \cite[Thm.~1.1]{Spivakovsky}, regularity implies pro-smoothness.)
    \end{exam}

    \begin{exam}
      For $X$ a singular scheme of dimension $d$, \thmref{thm:Weibel} was known as Weibel's conjecture (formulated in \cite[Questions~2.9]{WeibelAnalytic} for $X$ affine).
      Weibel's conjecture was proven by Kerz--Strunk--Tamme \cite[Thm.~B]{KerzStrunkTamme} for schemes.
    \end{exam}

    \begin{rem}
      \thmref{thm:Weibel} also holds for Deligne--Mumford stacks, or more generally Artin stacks with quasi-finite diagonal; see \cite[Thm.~D, Rem.~5.3.3]{BachmannKhanRaviSosnilo}.
      In fact, there is also a statement for algebraic stacks with affine diagonal and nice stabilizers, but Krull dimension has to be replaced by a weaker bound.
    \end{rem}

    \begin{proof}[Proof of \thmref{thm:Weibel}]
      By assumption, there exists a noetherian algebraic space $S$ of dimension $d$ and a cofiltered system of smooth $S$-schemes $(X_\alpha)_\alpha$ with affine transition maps, whose limit is $X$.
      By continuity (\thmref{thm:contin K}) there are canonical isomorphisms
      \[ \K_{-n}(X) \simeq \colim_\alpha \K_{-n}(X_\alpha) \]
      for all $n\in\bZ$.
      Hence \itemref{item:Weibel/van} would follow from the analogous claims for $X_\alpha$.
      Similarly for \itemref{item:Weibel/htp}, since by \corref{cor:continlocfr} the vector bundle $\pi : E \to X$ descends to $X_\alpha$ for every $\alpha$.

      Thus we may assume that $X$ is smooth over $S$.
      In this case, the statement is proven by a slight variant on the proof of Weibel's conjecture.
      See \cite[Rem.~5.3.3]{BachmannKhanRaviSosnilo}, or \cite{Sadhu}\footnote{%
        Thanks to the referee for pointing out this reference.
      } in the case of schemes.
    \end{proof}

  \subsection{Homotopy invariant K-theory}
  \label{ssec:KH}

    We can introduce a variant of K-theory by \emph{forcing} homotopy invariance for all (possibly singular) spaces.
    This construction goes back to Weibel \cite{WeibelKH}.
    See \cite{KrishnaRavi,HoyoisKrishna,HoyoisKH} for the case of stacks, \cite{KhanKblow} for the case of derived algebraic spaces, and \cite{KhanRavi} for derived stacks.
    We restrict to algebraic spaces for simplicity of exposition.
    
    Let $\Delta^\bullet$ denote the standard cosimplicial affine scheme whose $n$th term is affine $n$-space $\A^n$ (see e.g. \cite[p.~45]{MorelVoevodsky}).
    For any qcqs derived algebraic space $X$, let $\KH(X)$ denote the geometric realization of the simplicial diagram $\KB(\Delta^\bullet)$:
    \[ \KH(X) = \lim_{[n]\in\bDelta^\op} \KB(X \times \Delta^n). \]
    The canonical map of presheaves $\KB \to \KH$ exhibits $\KH$ as the $\A^1$-localization of K-theory in the sense of $\A^1$-homotopy theory (see \cite{CisinskiKH}).

    \begin{rem}
      If $X$ has bounded structure sheaf then by the universal property of $\A^1$-localization and homotopy invariance for G-theory (\corref{cor:G htp vb}), we find that the canonical map $\KB(X) \to \G(X)$ (\thmref{thm:Poincare}) factors through a map
      \[ \KH(X) \to \G(X), \]
      functorial with respect to finite Tor-amplitude inverse images.
      If $X$ is regular, then so are $X\times \Delta^n$ for all $n$, so this map is invertible.
      Hence all the maps
      \[ \K(X) \to \KB(X) \to \KH(X) \to \G(X) \]
      are invertible when $X$ is regular.
    \end{rem}

    \begin{rem}\label{rem:KH basics}
      Since colimits of spectra are exact, $\KH$ inherits the following properties from $\KB$: localization sequence (\thmref{thm:K localization}), excision of closed subsets (\thmref{thm:K excision}), projective bundle formula (\thmref{thm:K projbun}), and blow-up formula (\thmref{thm:K blow}).
    \end{rem}

    \begin{rem}
      For any $\bZ$-linear stable \inftyCat $\sA$ we can define $\KH(\sA)$ as the spectrum
      \[ \KH(\sA) = \colim_{n\in\bDelta^\op} \KB(\sA \otimes \Dperf(\Delta^n)) \]
      so that $\KH(X) = \KH(\Dperf(X))$.
      Therefore $\KH$ also has cup products and direct images along proper morphisms that are almost of finite presentation and of finite Tor-amplitude.
      These satisfy the projection and base change formulas (Propositions~\ref{prop:K proj} and \ref{prop:K base change}).
    \end{rem}

  \subsection{Properties of \texorpdfstring{$\KH$}{KH}}

    As hinted at earlier, passing from $\KB$ to $\KH$ has the remarkable side effect of forcing several other properties that were only true in K-theory either for regular (nonsingular) stacks, or only for very low $K$-groups for affines.

    \begin{prop}[Homotopy invariance]
      For any qcqs derived algebraic space $X$ and any vector bundle $\pi : E \to X$, the map $\pi^* : \KH(X) \to \KH(E)$ is invertible.
    \end{prop}
    \begin{proof}
      Using \thmref{thm:K excision}\itemref{item:K excision etale}, we may assume the bundle is trivial.
      By induction we may assume it is of rank one.
      In that case this holds by construction.
    \end{proof}

    \begin{thm}[Nil-invariance]\label{thm:KH derived invariance}
      Let $i : Z \to X$ be a surjective closed immersion of qcqs derived algebraic spaces.
      Then the map $i^* : \KH(X) \to \KH(Z)$ is invertible.
    \end{thm}
    \begin{proof}
      See \cite[Thm.~5.13]{KhanKblow}.
    \end{proof}

    In particular, taking $Z$ to be the classical truncation of $X$ shows that $\KH$ is insensitive to derived structures.

    We also have the following versions of ``excising open subsets''.

    \begin{thm}[Milnor excision]\label{thm:KH Milnor}
      For any Milnor square of qcqs derived algebraic spaces as in \ssecref{ssec:excising opens}, the induced square
      \[ \begin{tikzcd}
        \KH(X)\ar{r}{i^*}\ar{d}{f^*}
        & \KH(Z)\ar{d}
        \\
        \KH(X')\ar{r}
        & \KH(Z')
      \end{tikzcd} \]
      is cartesian.
    \end{thm}

    \begin{proof}
      By \thmref{thm:KH derived invariance} we may replace the square by its classical truncation.
      By \thmref{thm:K excision}\itemref{item:K excision etale} and \cite[Thm.~3.4.2.1]{LurieSAG} we may assume it consists of affine schemes.
      In this case the claim is a classical result of Weibel (it is equivalent to \cite[Thm.~2.1]{WeibelKH}).
    \end{proof}

    \begin{thm}[Proper excision]\label{thm:KH proper excision}
      Let $f : X' \to X$ be a proper morphism of qcqs derived algebraic spaces.
      Let $U \sub X$ and $U' \sub X'$ be quasi-compact open subsets such that $f$ restricts to an isomorphism $U'_\red \to U_\red$, and let $Z \sub X$ and $Z' \sub X'$ be their reduced complements.
      Then the square
      \[ \begin{tikzcd}
        \KH(X)\ar{r}\ar{d}{f^*}
        & \KH(Z)\ar{d}
        \\
        \KH(X')\ar{r}
        & \KH(Z')
      \end{tikzcd} \]
      is cartesian.
    \end{thm}
    \begin{proof}
      For schemes this is the main result of \cite{CisinskiKH} (where the noetherian assumption can be dropped by \cite[App.~C]{HoyoisLefschetz}).
      For the case of algebraic spaces see \cite[Thm.~D]{KhanKblow}.
    \end{proof}

    Finally we also get the following Mayer--Vietoris property for closed subspaces.

    \begin{cor}
      Let $X$ be a qcqs derived algebraic space.
      For any closed subspaces $Y$ and $Z$ such that $X = Y \cup Z$, the square
      \[ \begin{tikzcd}
        \KH(X) \ar{r}\ar{d}
        & \KH(Y)\ar{d}
        \\
        \KH(Z)\ar{r}
        & \KH(Y\cap Z)
      \end{tikzcd} \]
      is cartesian.
      Here $Y\cap Z$ is the classical or derived scheme-theoretic intersection ($\KH$ is insensitive to the difference by \thmref{thm:KH derived invariance}).
    \end{cor}

\section{Rational étale K-theory and G-theory}
\label{sec:Get}

  Given a \dA stack $\sX$, let $\KB(\sX)_\Q$ and $\G(\sX)_\Q$ denote the rationalized K- and G-theory spectra, respectively.
  In this section we introduce étale-localized versions $\Ket(\sX)_\Q$ and $\Get(\sX)_\Q$.

  \subsection{Descent on algebraic spaces}

    \begin{thm}[Étale and flat descent]\label{thm:etale descent asp}\leavevmode
      \begin{thmlist}
        \item
        On the \inftyCat of qcqs derived algebraic spaces, the presheaf $X \mapsto \KB(X)_\Q$ satisfies descent for the étale and finite flat topologies.

        \item
        On the \inftyCat of qcqs derived algebraic spaces, the presheaf $X \mapsto \KH(X)_\Q$ satisfies descent for the fppf topology.

        \item
        On the \inftyCat of noetherian derived algebraic spaces and finite Tor-amplitude morphisms, the presheaf $X \mapsto \G(X)_\Q$ satisfies descent for the fppf topology.
      \end{thmlist}
    \end{thm}

    \begin{rem}
      Recall that any smooth surjection generates a covering for the étale topology.
      Thus in particular, \thmref{thm:etale descent asp} implies that for any smooth surjection $p : U \to X$, the inverse image map $p^*$ induces an isomorphism between $\KB(X)_\Q$ and the homotopy limit of the cosimplicial diagram
      \begin{equation*}
        \KB(U)_\Q
        \rightrightarrows \KB(U\fibprod_X U)_\Q
        \rightrightrightarrows \KB(U\fibprod_X U \fibprod_X U)_\Q
        \rightrightrightrightarrows \cdots
      \end{equation*}
      and similarly for KH and G-theory.
    \end{rem}

    \begin{proof}[Proof of \thmref{thm:etale descent asp}]
      We first show that each presheaf satisfies finite flat descent.
      For this we may restrict to the small étale site of a fixed derived algebraic space $S$.
      Combining \thmref{thm:K excision}\itemref{item:K excision etale} or its KH/G-theory analogue (\remref{rem:KH basics} or \corref{cor:G Nis}) with \cite[Thm.~3.4.2.1]{LurieSAG}, we may assume that $S$ is affine.
      Then the claim follows by applying \cite[Prop.~5.4]{ClausenMathewNaumannNoel} (taking $n=0$, cf.~Appendix~A in \emph{op. cit.}), which is a statement about localizing invariants of $\Dperf(S)$-linear \inftyCats (see Def.~A.1 of \emph{op. cit.}) that take values in rational spectra.
      For the K-theory statement, take the localizing invariant that sends a $\Dperf(S)$-linear \inftyCat $\sA$ to the spectrum $\KB(\sA)_\Q$.
      For $\KH$, take the localizing invariant
      \begin{equation*}
        \sA \mapsto \colim_{[n]\in\bDelta^\op} \KB(\sA \otimes_{\Dperf(S)} \Dperf(S \times \Delta^n))_\Q.
      \end{equation*}
      For G-theory, take the localizing invariant $\K'$ that sends a $\Dperf(S)$-linear \inftyCat $\sA$ to the spectrum $\K(\sA \otimes_{\Dperf(S)} \Coh(S))_\Q$.
      Indeed, for every $X$ étale over $S$, there is a canonical isomorphism
      \begin{align*}
        \K'(\Dperf(X)) &= \K(\Dperf(X) \otimes_{\Dperf(S)} \Coh(S))_\Q\\
        &\simeq \K(\Coh(X))_\Q = \G(X)_\Q
      \end{align*}
      by \cite[Chap.~4, Rem.~3.3.3]{GaitsgoryRozenblyum}.

      By \cite[Thm.~B.6.4.1]{LurieSAG}, combining finite flat descent with Nisnevich descent (\remref{rem:K descent}) yields étale descent for each presheaf.

      For $\KH_\Q$ and $\G_\Q$, we may use derived invariance (\thmref{thm:KH derived invariance} and \corref{cor:G derived invariance}) to restrict to the site of classical algebraic spaces.
      Since the fppf topology on this site is generated by Nisnevich coverings and finite flat coverings (see e.g.~\cite[Tag~\href{https://stacks.math.columbia.edu/tag/05WN}{05WN}]{Stacks}), the claim follows.
    \end{proof}

    \begin{thm}[Proper descent]\label{thm:G-theory/proper descent}
      On the \inftyCat of qcqs derived algebraic spaces, the presheaf $X \mapsto \KH(X)_\Q$ satisfies descent along proper schematic surjections of finite presentation.
      That is, for any proper schematic surjection of finite presentation $p : Y \to X$, $p^*$ induces an isomorphism between $\KH(X)_\Q$ and the homotopy limit of the cosimplicial diagram
      \begin{equation*}
        \KH(Y)_\Q
        \rightrightarrows \KH(Y\fibprodR_X Y)_\Q
        \rightrightrightarrows \KH(Y\fibprodR_X Y \fibprodR_X Y)_\Q
        \rightrightrightrightarrows \cdots.
      \end{equation*}
    \end{thm}

    \begin{proof}
      By derived invariance (\thmref{thm:KH derived invariance}) and descent (\corref{thm:etale descent asp}), we may restrict to the site of classical schemes.
      Furthermore we may fix a scheme $S$ and restrict to the site of $S$-schemes of finite presentation.
      Then we may apply the criterion of \cite[Thm.~2.9]{BhattScholze}, which reduces the claim to fppf descent (\thmref{thm:etale descent asp}) and proper cdh descent (\thmref{thm:KH proper excision}).
    \end{proof}

    \begin{rem}
      Presumably, \thmref{thm:G-theory/proper descent} extends to (non-schematic) proper surjections of finite presentation, but I did not check that \cite[Thm.~2.9]{BhattScholze} extends to algebraic spaces.
    \end{rem}

    \begin{rem}
      In fact, the previous proof shows that $\KH_\Q$ satisfies descent for Voevodsky's h-topology (or rather, Rydh's non-noetherian generalization of it \cite[Sect.~8]{RydhH}) on the site of (derived) schemes.
    \end{rem}

  \subsection{Étale K-theory and G-theory}

    \thmref{thm:etale descent asp} does not extend to stacks.
    In fact, on \dA stacks, rational G-theory only satisfies descent with respect to étale maps that are \emph{isovariant} or ``stabilizer-preserving'' (see \cite[Sect.~3]{JoshuaRR}).
    
    Nevertheless, \thmref{thm:etale descent asp} implies that $\KB(-)_\Q$, $\KH(-)_\Q$, and $\G(-)_\Q$ extend uniquely to étale sheaves on \dA stacks, which we denote by
    \[ \Ket(-)_\Q, ~ \KHet(-)_\Q, ~\text{and}~ \Get(-)_\Q. \]
    These could equivalently described as the étale localizations of the presheaves $\KB(-)_\Q$, $\KH(-)_\Q$, and $\G(-)_\Q$, respectively.
    Note that, as before, we only ever discuss descent with respect to \v{C}ech covers and not hypercovers.

    Thus by construction, for any \dA stack $\sX$ and any smooth atlas $p : X \to \sX$, $\Ket(\sX)_\Q$ is isomorphic to the homotopy limit of the cosimplicial diagram
    \begin{equation*}
      \Ket(X)_\Q
      \rightrightarrows \KB(X\fibprod_\sX X)_\Q
      \rightrightrightarrows \KB(X\fibprod_\sX X \fibprod_\sX X)_\Q
      \rightrightrightrightarrows \cdots
    \end{equation*}
    and similarly for KH and G-theory.
    For example, if $\sX = [X/G]$ is the quotient of a derived algebraic space $X$ by the action of an fppf group space $G$, then $\Ket([X/G])_\Q$ is the Borel-type $G$-equivariant K-theory of $X$:
    \begin{equation*}
      \Ket(X)_\Q
      \rightrightarrows \KB(G \times X)_\Q
      \rightrightrightarrows \KB(G \times G \times X)_\Q
      \rightrightrightrightarrows \cdots
    \end{equation*}
    Similarly for $\KHet(-)_\Q$ and $\Get(-)_\Q$.

    All properties of K-theory, G-theory, and KH-theory involving \emph{inverse} image functoriality easily extend to the étale-local variants just by descent from the case of algebraic spaces.
    This applies for example to the projective bundle formula, excision for quasi-smooth blow-ups, and homotopy invariance.
    In particular, we can drop the extra hypothesis in the latter statement (compare with \corref{cor:G htp vb stack}).

    \begin{cor}[Homotopy invariance]\label{cor:G-theory/htp Get}
      Let $\sX$ be a \dA stack.
      Let $\sE$ be a perfect complex on $\sX$ of Tor-amplitude $\le 0$, and $\pi : \bV_\sX(\sE) \to \sX$ the associated vector bundle stack.
      Then inverse image along $\pi$ induces isomorphisms of spectra
      \begin{align*}
        \pi^* : \KHet(\sX)_\Q \to \KHet(\bV_\sX(\sE))_\Q,\\
        \pi^* : \Get(\sX)_\Q \to \Get(\bV_\sX(\sE))_\Q.
      \end{align*}
    \end{cor}
    \begin{proof}
      One argues as in \corref{cor:G htp vb stack}, except that the resolution property can be guaranteed simply by localizing on $\sX$ (in the étale topology).
    \end{proof}

    In contrast, the direct image functoriality is somewhat subtle (see further discussion in \ssecref{ssec:G f_*} below).
    For any proper \emph{representable} morphism $f : \sX \to \sY$ between noetherian \dA stacks, there is indeed a direct image map
    \[ f_* : \Get(\sX)_\Q \to \Get(\sY)_\Q. \]
    Since properness is local on the target, descent allows us to reduce to the case where $\sY = Y$ is representable by a derived algebraic space.
    Since $f$ is representable, it follows that $\sX = X$ is also an algebraic space.
    Thus we simply take the map
    $$\Get(\sX)_\Q \simeq \G(X)_\Q \xrightarrow{f_*} \G(Y)_\Q \simeq \Get(\sY)_\Q.$$
    Properties of G-theory involving proper representable direct images also extend to the étale-local variants (e.g., localization, base change, projection formulas).

  \subsection{Proper codescent}

    Just as étale excision (\corref{cor:G Nis}) in G-theory extends to étale descent in rational G-theory (\thmref{thm:etale descent asp}), proper co-excision (\thmref{thm:proper co-excision G}) similarly extends to proper co-descent with rational coefficients.
    This extends moreover to rational étale G-theory of stacks.

    \begin{thm}[Proper codescent]\label{thm:proper codescent}
      On the \inftyCat of noetherian \dA stacks, the copresheaf $\sX \mapsto \Get(\sX)_\Q$ satisfies descent along proper representable surjections.
    \end{thm}

    \begin{rem}
      More concretely, \thmref{thm:proper codescent} asserts that for any proper representable surjection $f : \sZ \to \sX$, $f_*$ induces an isomorphism between $\Get(\sX)_\Q$ and the homotopy colimit of the simplicial diagram
      \begin{equation*}
        \cdots
        \rightrightrightrightarrows \Get(\sZ\fibprodR_\sX \sZ \fibprodR_\sX \sZ)_\Q
        \rightrightrightarrows \Get(\sZ\fibprodR_\sX \sZ)_\Q
        \rightrightarrows \Get(\sZ)_\Q.
      \end{equation*}
      If $\sX$ and $\sZ$ are classical, then the fibred products need not be derived (see \corref{cor:G derived invariance}).
    \end{rem}

    We will need a few preliminary observations.
    The first is in fact a straightforward corollary of \thmref{thm:proper codescent} (consider the diagonal, which is a surjective closed immersion), but we can also give a direct proof.

    \begin{prop}[Topological invariance]\label{prop:topinv}
      Let $f : \sX \to \sY$ be a finite radicial surjection of noetherian \dA stacks (i.e., a universal homeomorphism of finite type).
      Then the induced morphism
      \[ f_* : \Get(\sX)_\Q \to \Get(\sY)_\Q \]
      is invertible.
    \end{prop}
    \begin{proof}
      By derived invariance (\corref{cor:G derived invariance}), we may assume $\sX$ and $\sY$ are classical.
      By étale descent (and base change, \propref{prop:G base change}) we may also assume $\sY = Y$ (and hence $\sX = X$) is affine.
      By \cite[IV\textsubscript{4}, 17.16.4]{EGA}, there is a partition of $Y$ by finitely many locally closed subschemes $Y_i$ along with finite flat surjections $Y'_i \twoheadrightarrow Y_i$ such that each base change $f_i : X \fibprod_Y Y'_i \to Y'_i$ admits a section.
      By localization (\thmref{thm:G localization}), we may replace $Y$ by $Y_i$ and by finite flat descent (\thmref{thm:etale descent asp}) we may further replace $Y_i$ by $Y'_i$.
      In other words, we may assume that $f$ admits a section $s$.
      Since $f$ is radicial, $s$ is a surjective closed immersion.
      Then $s_*$ is an isomorphism (\corref{cor:G nil-invariance}), hence so is $f_*$.
    \end{proof}

    \begin{proof}[Proof of \thmref{thm:proper codescent}]
      Let $f : \sZ \to \sX$ be a proper representable surjection.
      By nil-invariance (\corref{cor:G nil-invariance}), we may assume $\sX$ and $\sZ$ are reduced classical stacks.
      By étale descent (and base change, \propref{prop:G base change}) we may assume $\sX = X$ is affine (and hence $\sZ = Z$ is an algebraic space).
      By noetherian induction and localization (\thmref{thm:G localization}), we may further reduce to the case where $X$ is the spectrum of a field $k$.
      Let $x$ be a closed point of the (nonempty, noetherian) scheme $X$.
      Its residue field $\kappa(x)$ is purely inseparable over a Galois extension $k'$ of $k$.
      Finite étale descent (\thmref{thm:etale descent asp}) in the case of $\Spec(k') \to \Spec(k)$ gives an isomorphism
      \[ \G(k')_\Q^G \simeq \G(k), \]
      where the left-hand side is the spectrum of homotopy invariants with respect to the Galois group $G$.
      Similarly for the induced Galois cover $Z' = Z \fibprod_{X} \Spec(k') \to Z$.
      Thus the diagram
      \begin{equation*}
        \cdots
        \rightrightrightrightarrows \Get(Z\fibprodR_{X} Z \fibprodR_{X} Z)_\Q
        \rightrightrightarrows \G(Z\fibprodR_{X} Z)_\Q
        \rightrightarrows \G(Z)_\Q
        \rightarrow \G(X)_\Q
      \end{equation*}
      is identified with the colimit of the same diagram where $Z$ and $X = \Spec(k)$ are replaced by $Z'$ and $X' := \Spec(k')$.
      Since formation of homotopy invariants commutes with colimits (recall that with rational coefficients, the former coincide with homotopy coinvariants, which are colimits), it will suffice to show the claim after this replacement.
      By \propref{prop:topinv} we may also base change further along the finite radicial surjection $\Spec(\kappa(x)) \to \Spec(k')$.
      In particular, we may assume that $f : Z \to X$ admits a \emph{section} $s$.
      In that case the claim is obvious, as $s$ induces a splitting of the augmented simplicial object in question (which is then automatically a homotopy colimit diagram, see e.g. \cite[6.1.3.16]{LurieHTT}).
    \end{proof}

  \subsection{Non-representable direct image}
  \label{ssec:G f_*}

    Let $f : \sX \to \sY$ be a proper morphism between noetherian \dA stacks.
    We can use proper codescent to construct a direct image map
    \[ f_* : \Get(\sX)_\Q \to \Get(\sY)_\Q \]
    even if $f$ is not representable.

    By \cite[Thm.~1.1]{Olsson} (applied to the classical truncations), there exists a scheme $Z$ and a proper surjection $g : Z \to \sX$ such that the composite $Z \to \sX \to \sY$ is projective (in particular, proper representable).
    By proper codescent along $g$ (\thmref{thm:proper codescent}), we get an essentially unique map $f_* : \Get(\sX)_\Q \to \Get(\sY)_\Q$ extending the direct image maps
    \[ \Get(Z \fibprod_\sX \cdots \fibprod_\sX Z)_\Q \to \Get(\sY)_\Q. \]

    Given another choice of $Z$ and $g$, it is an easy exercise to show that the two morphisms $f_*$ are homotopic.
    Similarly, if $f : \sX \to \sY$ and $f' : \sY \to \sZ$ are two proper morphisms (where $\sX$ and $\sY$ both have quasi-finite separated diagonal), then the two maps $f'_* \circ f_*$ and $(f'\circ f)_*$ are homotopic.
    (We do not address coherence of these homotopies here.)
    These direct images also satisfy the base change and projection formulas (Propositions~\ref{prop:G projection} and \ref{prop:G base change}).

    \begin{warn}\label{warn:G f_*}
      These non-representable direct images are \emph{not} compatible with those in ``genuine'' G-theory induced by proper push-forward of coherent complexes.
      That is, let $f : \sX \to \sY$ be a proper morphism of finite cohomological dimension and assume $\sX$ has quasi-finite separated diagonal.
      Then the square
      \[ \begin{tikzcd}
        \G(\sX)_\Q \ar{r}{f_*}\ar{d}
        & \G(\sY)_\Q \ar{d}
        \\
        \Get(\sX)_\Q \ar{r}{f_*}
        & \Get(\sY)_\Q
      \end{tikzcd} \]
      typically does not commute (unless $f$ is representable).
      
      Consider the following standard example.
      Let $k$ be a field, $G$ a finite group, and $f : BG \to \Spec(k)$ the structural map of the classifying stack.
      If the order of $G$ is prime to the characteristic of $k$, then $f$ is of finite cohomological dimension.
      Let $g : \Spec(k) \twoheadrightarrow BG$ denote the quotient morphism, which is finite étale.
      Using the base change and projection formulas, one finds that both composites $g^*g_*$ and $g_*g^*$ on $\G(k)$ are multiplication by the order of $G$.
      In particular, $g$ induces an isomorphism
      $$\Q \simeq \Get(k)_\Q \xrightarrow{g_*} \Get(BG)_\Q.$$
      By functoriality, it follows that $f_* : \Get(BG)_\Q \to \Get(k)_\Q$ is homotopic to a scalar multiple of $g^*$.
      Since the étale localization map $\G(BG)_\Q \to \Get(BG)_\Q$ commutes with $g^*$ by definition, we see that the square commutes if and only if $f_*$ and $g^*$ are also homotopic at the level of $\G(-)_\Q$ (up to multiplication by a scalar).
      But under the identifications of $\G(BG)_\Q$ with the K-theory spectrum of $G$-equivariant $k$-vector spaces and $\G(k)_\Q$ with the K-theory spectrum of $k$-vector spaces, the map $g^*$ sends the class of a $G$-representation $V$ to the underlying vector space $V$ while $f_*$ sends it to the class of the invariant subspace $V^G$.
      No scalar multiplication identifies these (e.g. take $V$ to be a nontrivial representation of dimension one so that $V^G=0$).
    \end{warn}

\section{Virtual fundamental classes}
\label{sec:fund}

  \subsection{Fundamental classes in G-theory}

    Let $\sX$ be a regular (nonsingular) Artin stack.
    K-theoretic Poincaré duality (\thmref{thm:Poincare}) can be reformulated as the assertion that there is a canonical class $[\sX]_{\G} \in \G(\sX)$ such that cap product gives an isomorphism
    \[ [\sX]_{\G} \cap (-) : \KB(\sX) \to \G(\sX). \]
    Of course, $[\sX]_{\G}$ is just the class of the structure sheaf $[\sO_\sX]$.

    This construction can be extended to (derived) stacks that are possibly \emph{singular} but still quasi-smooth.
    First note that by \lemref{lem:qsm fTa} we have:

    \begin{cor}\label{cor:qsm bounded}
      Let $\sX \to \sY$ be a quasi-smooth morphism of \dA stacks.
      If $\sY$ has bounded structure sheaf, then so does $\sX$.
    \end{cor}

    \begin{constr}[G-theoretic fundamental class]\label{constr:G fund}
      Let $\sX$ be a quasi-smooth \dA stack over a regular noetherian Artin stack $\sS$.
      Then $\sX$ is noetherian and its structure sheaf $\sO_\sX$ is coherent by \corref{cor:qsm bounded} and \remref{rem:Perf in Coh}.
      In particular, there is a \emph{G-theoretic fundamental class}
      \begin{equation*}
        [\sX]_{\G} := [\sO_\sX] \in \G(\sX).
      \end{equation*}
    \end{constr}

    \begin{rem}\label{rem:piospq}
      Under the canonical isomorphism $\G(\sX) \simeq \G(\sX_\cl)$ (\corref{cor:G derived invariance}), the fundamental class $[\sX]_{\G}$ corresponds uniquely to a class
      \begin{equation*}
        [\sX]^\vir_{\G} \in \G(\sX_\cl),
      \end{equation*}
      which we call the G-theoretic \emph{virtual} fundamental class.
      Explicitly, it can be described as the alternating sum
      \[ [\sX]^\vir_{\G} = \sum_{i\le 0} (-1)^i [\H^i(\sO_\sX)], \]
      which is finite by assumption, where $\H^i(\sO_\sX) = \pi_{-i}(\sO_\sX)$ are viewed as coherent sheaves on $\sX_\cl$ via the equivalence $\Coh(\sX_\cl) \simeq \Dcoh(\sX)^\heartsuit$.
    \end{rem}

    This virtual G-theory class agrees with the \emph{virtual structure sheaf} studied in detail by Y.-P. Lee \cite{Lee}.
    \constrref{constr:G fund} goes back to \cite[1.4.2]{Kontsevich} and was written down more precisely by To\"en in \cite[\S 4.4, para.~3]{ToenOverview}.
    Note that the various properties of virtual structure sheaves \cite[Subsect.~2.4]{Lee} are completely transparent from \constrref{constr:G fund}.

    \begin{rem}
      Since the structural morphism $f : \sX \to \sS$ is quasi-smooth and hence of finite Tor-amplitude (\lemref{lem:qsm fTa}), there is an inverse image map
      \[ f^* : \G(\sS) \to \G(\sX). \]
      and the fundamental class $[\sX]_{\G}$ is of course the same as the image of the fundamental class $[\sO_\sS] \in \G(\sS)$ defined above.
    \end{rem}

    \begin{rem}
      More generally, for any quasi-smooth morphism of noetherian \dA stacks $f : \sX \to \sY$ there is again by \lemref{lem:qsm fTa} a map
      \[ f^* : \G(\sY) \to \G(\sX). \]
      which corresponds under the isomorphisms of \corref{cor:G derived invariance} to a map
      \[ \G(\sY_\cl) \to \G(\sX_\cl) \]
      which can be thought of as a ``virtual pullback'' or virtual Gysin map along the morphism of classical stacks $f_\cl : \sX_\cl \to \sY_\cl$.
      It agrees with the virtual pullbacks of \cite{Qu} (as will follow from the next subsection).
    \end{rem}

  \subsection{Fundamental classes via deformation to the normal stack}

    In this subsection we give an alternative construction of the maps $f^* : \G(\sY) \to \G(\sX)$ in the case when $f : \sX \to \sY$ is quasi-smooth.
    This description was implicitly used in the proof of the Grothendieck--Riemann--Roch formula proven in \cite{KhanVirtual} (see \ssecref{ssec:GRR} below).

    Recall that for any quasi-smooth morphism $f : \sX \to \sY$ there is a canonical ``deformation to the normal stack'' which sits in a commutative diagram
    \begin{equation}\label{eq:D}
      \begin{tikzcd}
        \sX \ar{r}\ar{d}{0}
        & \sX \times \A^1 \ar{d}
        & \sX \times \bG_m \ar{d}{f\times\id}\ar{l}
        \\
        N_{\sX/\sY} \ar{r}{\hat{i}}\ar{d}{v}
        & D_{\sX/\sY}\ar{d}{u}
        & \sY\times\bG_m \ar[swap]{l}{\hat{j}}\ar[equals]{d}
        \\
        \sY \ar{r}{i}
        & \sY\times\A^1
        & \sY\times\bG_m\ar[swap]{l}{j}
      \end{tikzcd}
    \end{equation}
    This is a family of quasi-smooth morphisms $\sX \times \A^1 \to D_{\sX/\sY}$ parametrized by the affine line whose generic fibre is the morphism $f : \sX \to \sY$ and whose special fibre is the zero section $0 : \sX \to N_{\sX/\sY}$ of the normal stack.
    The latter is by definition the vector bundle stack $\bV(\sL_{\sX/\sY}[-1])$ associated to the $(-1)$-shifted cotangent complex.
    When $f$ is a closed immersion, $N_{\sX/\sY}$ is just the normal bundle and $D_{\sX/\sY}$ is the usual deformation to the normal bundle.
    Note that the left-hand arrow $v$ is the composite of the projection $\pi : N_{\sX/\sY} \to \sX$ with $f : \sX \to \sY$.
    See \cite[Thm.~1.3]{KhanVirtual}.

    Using this deformation we can construct a specialization map in G-theory in parallel to the case of Borel--Moore homology in \cite{KhanVirtual}.
    Consider the localization triangle (\thmref{thm:G localization}) associated to the lower row of \eqref{eq:D}:
    \begin{equation*}
      \G(\sX\times\bG_m)[-1]
      \xrightarrow{\partial} \G(\sX)
      \xrightarrow{i_*} \G(\sX \times \A^1)
      \xrightarrow{j^*} \G(\sX \times \bG_m).
    \end{equation*}
    The boundary map $\partial$ admits a section
    \begin{equation*}
      \gamma_b : \G(\sX)
      \xrightarrow{q^*} \G(\sX\times\bG_m)
      \xrightarrow{\cap b} \G(\sX\times\bG_m)[-1]
    \end{equation*}
    where $q : \sX\times\bG_m\to\sX$ is the projection and $b \in \K(\sX \times \bG_m)[-1]$ is the Bott class (inverse image of $b \in \K(\bZ[T^{\pm 1}])[-1]$).

    Consider also the localization triangle associated to the middle row of \eqref{eq:D}.
    These fit into a commutative diagram
    \begin{equation*}
      \begin{tikzcd}
        \G(\sY\times\bG_m)[-1] \ar{r}{\partial}\ar[equals]{d}
          & \G(\sY) \ar{r}{i_*}\ar{d}{v^*}
          & \G(\sY\times\A^1) \ar{r}{j^*}\ar{d}{u^*}
          & \G(\sY\times\bG_m) \ar[equals]{d}
        \\
        \G(\sY\times\bG_m)[-1] \ar{r}{\hat{\partial}}
          & \G(N_{\sX/\sY}) \ar{r}{\hat{i}_*}
          & \G(D_{\sX/\sY}) \ar{r}{\hat{j}^*}
          & \G(\sY\times\bG_m)
      \end{tikzcd}
    \end{equation*}
    where $\partial$ and $\hat{\partial}$ are the respective boundary maps.
    The right-hand square commutes by functoriality of inverse image, the middle square commutes by the base change formula, and the left-hand square commutes as a consequence.

    \begin{defn}
      Let $f : \sX \to \sY$ be a quasi-smooth morphism of noetherian \dA stacks.
      The \emph{specialization map} is the composite
        \begin{equation*}
          \sp_{\sX/\sY} : \G(\sY)
            \xrightarrow{\gamma_b} \G(\sY\times\bG_m)[-1]
            \xrightarrow{\hat{\partial}} \G(N_{\sX/\sY}),
        \end{equation*}
      where the notation is as above.
    \end{defn}

    Since $\gamma_b$ is a section of $\partial$, commutativity of the diagram above immediately yields a canonical identification
    \[ \sp_{\sX/\sY} = v^* = \pi^* \circ f^*, \]
    where $\pi : N_{\sX/\sY} \to \sX$ is the projection.
    As long as we assume $\sX$ has affine stabilizers (e.g. is a derived Deligne--Mumford stack), we know that the map $\pi^*$ is invertible by homotopy invariance (\corref{cor:G-theory/htp Get}).
    Thus we find that the inverse image map $f^*$ can be described in terms of the specialization map:

    \begin{prop}\label{prop:f^* via sp in Get}
      Let $f : \sX \to \sY$ be a quasi-smooth morphism of noetherian \dA stacks with affine stabilizers.
      Then $f^* : \G(\sY) \to \G(\sX)$ is canonically homotopic to the composite
      \begin{equation*}
        \G(\sY)
          \xrightarrow{\sp_{\sX/\sY}} \G(N_{\sX/\sY})
          \xrightarrow{(\pi^*)^{-1}} \G(\sX).
      \end{equation*}
    \end{prop}

    \begin{rem}
      Throughout the above discussion we can replace G-theory by rational étale G-theory.
      In that case we do not need to assume affineness of stabilizers, since homotopy invariance for vector bundle stacks always holds in $\Get$ (\corref{cor:G-theory/htp Get}).
    \end{rem}

  \subsection{Grothendieck--Riemann--Roch formulas}
  \label{ssec:GRR}

    Let $k$ be a field and $\sX$ a \dA stack of finite type over $k$.
    Via stable motivic homotopy theory (see \cite{Riou}, \cite{KhanVirtual}) one can construct canonical isomorphisms
    \[ \tau_\sX : \Get(\sX)_\Q \to \bigoplus_{n\in\Z} \RGamma^{\mrm{BM}}(\sX, \bQ(n))[-2n], \]
    where on the right-hand side are spectra whose homotopy groups are the motivic Borel--Moore homology groups.
    Combining with the étale localization map $\G(\sX)_\Q \to \Get(\sX)_\Q$, we get a map
    \[ \tau_\sX : \G(\sX)_\Q \to \bigoplus_{n\in\Z} \RGamma^{\mrm{BM}}(\sX, \bQ(n))[-2n] \]
    which on $\pi_0$ gives a map
    \[ \tau_\sX : \G_0(\sX)_\Q \to \H^{\mrm{BM}}_{2n}(\sX, \Q(n)) \simeq \on{A}_*(\sX_\cl)_\Q \]
    where, if $\sX$ has affine stabilizers, the target is now identified with the Chow groups of the classical truncation $\sX_\cl$.
    This is a derived and stacky extension of the Baum--Fulton--MacPherson transformation \cite{BaumFultonMacPherson} and we have the following analogue of the main result of \emph{op. cit}.

    \begin{thm}\label{thm:GRR f_*}\leavevmode
      \begin{thmlist}
        \item
        Let $f : \sX \to \sY$ be a proper morphism between \dA stacks over $k$.
        Then there is a commutative square
        \[ \begin{tikzcd}
          \Get_0(\sX)_\Q\ar{d}{\tau_\sX}\ar{r}{f_*}
          & \Get_0(\sY)_\Q \ar{d}{\tau_\sY}
          \\
          \H^{\mrm{BM}}_{2n}(\sX, \Q(n))\ar{r}{f_*}
          & \H^{\mrm{BM}}_{2n}(\sY, \Q(n)).
        \end{tikzcd} \]

        \item\label{item:GRR f_* rep}
        Let $f : \sX \to \sY$ be a proper \emph{representable} morphism between \dA stacks over $k$.
        Then there is a commutative square
        \[ \begin{tikzcd}
          \G_0(\sX)_\Q\ar{d}{\tau_\sX}\ar{r}{f_*}
          & \G_0(\sX)_\Q\ar{d}{\tau_\sY}
          \\
          \H^{\mrm{BM}}_{2n}(\sX, \Q(n))\ar{r}{f_*}
          & \H^{\mrm{BM}}_{2n}(\sY, \Q(n)).
        \end{tikzcd} \]
      \end{thmlist}
    \end{thm}
    \begin{proof}
       For the first claim, one may identify $\Get_0(\sX)$ with $\RGamma^{\mrm{BM}}(\sX, \mrm{KGL})$ in a way that is natural with respect to proper direct images (see \cite[Ex.~2.13]{KhanVirtual}).
       Then the claim follows immediately from the construction of proper direct images in Borel--Moore homology theories given in \cite{KhanVirtual}.
       The second claim follows by composing with the first with the square
       \[ \begin{tikzcd}
          \G_0(\sX)_\Q\ar{d}{\tau_\sX}\ar{r}{f_*}
          & \G_0(\sY)_\Q\ar{d}{\tau_\sY}
          \\
          \Get_0(\sX)_\Q \ar{r}{f_*}
          & \Get_0(\sY)_\Q
        \end{tikzcd} \]
       which commutes by construction when $f$ is representable (see \ssecref{ssec:G f_*}).
    \end{proof}

    \begin{rem}
      The representability hypothesis in claim~\itemref{item:GRR f_* rep} is necessary, see \warnref{warn:G f_*}.
      However, following \cite{ToenGRR}, it is possible to prove a variant where the lower horizontal map is replaced by the direct image of the induced morphism of inertia stacks.
      See \cite{KhanPortaRaviYu}.
    \end{rem}

    In \cite{KhanVirtual} a fundamental class $[\sX]$ is constructed in motivic Borel--Moore homology, when $\sX$ is quasi-smooth, and the following comparison with the G-theoretic fundamental class $[\sX]_{\G} = [\sO_\sX] \in \G(\sX)$ (conjectured in \cite[Question~4.7]{ToenOverview}) is proven.

    \begin{thm}\label{thm:GRR 1}
      For every quasi-smooth \dA stack $\sX$ over $k$, there is an equality
      \[ [\sX] = \Td_{\sX}^{-1} \cap\,\tau_\sX[\sX]_{\G}, \]
      where $\Td_{\sX}$ is the Todd class of the cotangent complex $\sL_\sX$.
    \end{thm}

    \begin{rem}
      Via the identification mentioned above we may view the fundamental class $[\sX]$ as a class $[\sX_\cl]^{\vir} \in \on{A}_*(\sX_\cl)_\Q$.
      For $\sX$ Deligne--Mumford, this coincides with the virtual fundamental class of Behrend--Fantechi \cite{BehrendFantechi} (with respect to obstruction theory on $\sX_\cl$ induced by the cotangent complex of $\sX$) as explained in \cite[\S 3.3]{KhanVirtual}.
      We can then read \thmref{thm:GRR 1} as a comparison with the G-theoretic virtual fundamental class (\remref{rem:piospq}) via the formula
      \[ [\sX_\cl]^\vir = (\Td_{\sX_\cl}^\vir)^{-1} \cap \tau_{\sX_\cl}\left([\sX_\cl]^\vir_{\G}\right) \]
      in $\on{A}_*(\sX_\cl)_\Q$, where $\Td^\vir_{\sX_\cl}$ is the Todd class of $\sL_{\sX}|_{\sX_\cl}$.
      This recovers the virtual Grothendieck--Riemann--Roch formulas of \cite{FantechiGoettsche,CiocanFontanineKapranov,LowreySchuerg}, all proven in the case where $\sX$ is a derived scheme admitting an embedding into a smooth ambient scheme, as well as the extension to the case of quotient stacks in \cite{RaviSreedhar}.
    \end{rem}
    
  \subsection{K-theoretic fundamental classes}
  \label{ssec:K fund}

    The following construction is a K-theory analogue of the ``cohomological fundamental class'' studied in \cite[\S 3.4]{KhanVirtual}.

    \begin{constr}
      Let $\sX$ be a \dA stack.
      For every quasi-smooth proper representable morphism $f : \sZ \to \sX$, there is a canonical class
      \begin{equation*}
        [\sZ/\sX]^{\K} := f_*(1) = [f_*(\sO_\sZ)] \in \KB(\sX),
      \end{equation*}
      where $f_* : \KB(\sZ) \to \KB(\sX)$ is the direct image map (which exists by \lemref{lem:qsm fTa}) and $1 = [\sO_\sZ] \in \KB(\sZ)$ is the unit.
      If $f$ is a closed immersion, then $[\sZ/\sX]^{\K}$ is supported on $\sZ$, i.e., it lives in $\KB(\sX~\mrm{on}~\sZ)$.
    \end{constr}

    When $\sX$ is nonsingular, this class is just the direct image of the G-theoretic fundamental class of $\sZ$:

    \begin{rem}
      Let $\sX$ be a regular Artin stack.
      For any quasi-smooth proper representable morphism $f : \sZ \to \sX$, consider the commutative square
      \[ \begin{tikzcd}
        \KB(\sZ) \ar{r}{f_*}\ar{d}
        & \KB(\sX) \ar{d}
        \\
        \G(\sZ) \ar{r}{f_*}
        & \G(\sX)
      \end{tikzcd} \]
      where the vertical arrows are the Cartan maps (\thmref{thm:Poincare}).
      Since $\sX$ is regular, the right-hand vertical arrow is invertible.
      Under this identification, it follows that the K-theoretic fundamental class $[\sZ/\sX]^\K \in \KB(\sX)$ can be identified with the direct image of the G-theoretic fundamental class $[\sZ]_\G \in \G(\sZ)$ (\constrref{constr:G fund}).
    \end{rem}

    \begin{prop}\label{prop:pinpqa}
      Suppose given a commutative square of \dA stacks
      \begin{equation*}
        \begin{tikzcd}
          \sZ' \ar{r}{i'}\ar{d}{p}
            & \sX' \ar{d}{q}
          \\
          \sZ \ar{r}[swap]{i}
            & \sX
        \end{tikzcd}
      \end{equation*}
      where $i$ and $i'$ are quasi-smooth closed immersions.
      If the square is homotopy cartesian, then we have a canonical identification
      \[ q^*[\sZ/\sX] \simeq [\sZ'/\sX'] \]
      in $\KB(\sX')$.
    \end{prop}
    \begin{proof}
      Evaluate the base change formula $q^*i_* \simeq i'_*p^*$ (\propref{prop:K base change}) on the unit $1 = [\sO_\sZ] \in \KB(\sZ)$.
    \end{proof}

    \begin{prop}[Non-transverse Bézout formula]\label{prop:Bezout}
      Let $\sX$ be a \dA stack.
      Let $f : \sY \to \sX$ and $g : \sZ \to \sX$ be quasi-smooth proper representable morphisms.
      Then we have
      \begin{equation*}
        [\sY/\sX]^{\K} \cup [\sZ/\sX]^{\K} \simeq [\sY \fibprodR_\sX \sZ/\sX]^{\K}
      \end{equation*}
      in $\KB(\sX)$.
    \end{prop}
    \begin{proof}
      By the projection and base change formulas (Propositions~\ref{prop:K proj} and \ref{prop:K base change}), we have
      \begin{equation*}
        f_*(1) \cup g_*(1)
        \simeq f_*f^*g_*(1)
        \simeq h_*(1)
      \end{equation*}
      where $h : \sY\fibprodR_\sX\sZ \to \sX$ is the projection.
    \end{proof}

    To deal with non-proper intersections there is an \emph{excess intersection formula} which expresses the failure of the base change formula (\propref{prop:pinpqa}) in terms of the K-theoretic Euler class of the excess bundle, see \cite{ThomasonBlow} and \cite{KhanExcess}.
    For example:

    \begin{prop}[Self-intersection formula]
      Let $i : \sZ \to \sX$ be a quasi-smooth closed immersion.
      Then there is a canonical identification
      \[ i^*i_*(-) \simeq e(\sN_{\sZ/\sX}) \cup (-) \]
      of maps $\KB(\sZ) \to \KB(\sZ)$.
      In particular,
      \[ i^*[\sZ/\sX] \simeq e(\sN_{\sZ/\sX}) \]
      in $\KB(\sZ)$ and \[ [\sZ/\sX] \cup [\sZ/\sX] \simeq [\sZ \fibprodR_\sX \sZ/\sX] \simeq i_*(e(\sN_{\sZ/\sX})) \]
      in $\KB(\sX)$ (or $\KB(\sX~\mrm{on}~\sZ)$).
    \end{prop}

  \subsection{The \texorpdfstring{$\gamma$}{gamma}-filtration and Chow cohomology of singular schemes}
    
    This subsection is a digression that attempts to justify my interest in these ``K-theoretic fundamental classes'' (\ssecref{ssec:K fund}), which do not seem to have received much attention in the literature so far.

    \subsubsection{The $\gamma$-filtration and Chow cohomology of nonsingular schemes}

      Let $X$ be a scheme of finite type over a field $k$ and consider the group $\on{Z}_*(X)$ of algebraic cycles on $X$.
      Recall that there is a map
      \[ \on{Z}_*(X) \to \G_0(X), \quad [Z] \mapsto [\sO_Z]. \]
      This map does not respect rational equivalence but does so if we pass from $\G_0(X)$ to the graded pieces of the coniveau or $\gamma$-filtration (both agree with rational coefficients); there are canonical surjections
      \[ \CH_i(X) \to \Gr^{d-i}(\G_0(X)), \quad [Z] \mapsto [\sO_Z] \]
      where $\CH_*(X)$ denotes the quotient of $\on{Z}_*(X)$ modulo rational equivalence.

      When $X$ is nonsingular, the Grothendieck--Riemann--Roch theorem implies that the kernel of this map is torsion.
      Moreover the $\gamma$-filtration is compatible with the cup product on $\K_0(X)$ and on the ``Poincaré dual'' theories there is an isomorphism of graded rings
      \begin{equation}\label{eq:oausb}
        \CH^*(X)_\Q \to \Gr^*(\K_0(X))_\Q.
      \end{equation}
      Here $\CH^*(X) = \CH_{d-*}(X)$ by definition and $\K_0(X) \simeq \G_0(X)$ by \thmref{thm:Poincare}.
      See \cite[Exp.~0, Appendix and Exp.~14, \S 4]{SGA6} and \cite[Ex.~15.3.6]{Fulton}.
    
    \subsubsection{Chow cohomology of singular schemes}

      For singular $X$ the right-hand side of \eqref{eq:oausb} still makes sense but it is less clear how to define ``Chow cohomology'' rings $\CH^*(X)$ such that \eqref{eq:oausb} holds.%
      \footnote{%
        Fulton's operational Chow groups \cite[Chap.~17]{Fulton} are commonly used as a substitute for a ``genuine'' Chow cohomology theory, but they satisfy neither this property nor many others (for example they do not even map to singular cohomology, see \cite{TotaroChow}).
        Voevodsky's motivic cohomology groups \cite{VoevodskyField} are almost what one is looking for except that rationally they compare to the \emph{homotopy invariant} version of K-theory.
        On singular schemes, motivic cohomology can be computed (as a presheaf of complexes) by the following procedure: restrict to smooth $k$-schemes, take the left Kan extension to $k$-schemes of finite type, and finally take the $\A^1$-homotopy localization of the result (see \cite[Prop.~5]{KhanCdh}).
        Stopping before the last step gives a non-homotopy invariant version of motivic cohomology.
        This is a variant of an old construction of Fulton from \cite[\S 3.1]{FultonSingular}.
        A similar procedure works for homotopy invariant algebraic K-theory (see e.g. \cite[Props.~5.2.2 and 5.3.7]{CisinskiKhan}) and in that case stopping before the last step does recover algebraic K-theory (see \cite[\S 3.2]{FultonSingular}, \cite[App.~A]{EHKSY3}).
        The derived Chow ring constructed below should hopefully compute these non-homotopy invariant motivic cohomology groups for nice enough schemes.
      }
      There is a reasonable definition of top-degree Chow cohomology $\CH^d(X)$, where $X$ is a singular $d$-dimensional variety, due to Levine and Weibel, see e.g. \cite{LevineWeibel,Srinivas,GuptaKrishna}.
      Here I want to briefly explain how the theory of K-theoretic fundamental classes suggests a construction of a ``derived'' Chow cohomology theory.

      Let $\mrm{Z}_\der^n(X)$ denote the free abelian group on quasi-smooth projective morphisms $f : Z \to X$ of relative virtual dimension $-n$.
      Denote the class of $f : Z \to X$ by $[Z/X] \in \mrm{Z}_\der^n(X)$.
      There is an obvious product
      \[ [Z / X] \cup [Z' / X] = [Z \fibprodR_X Z' / X], \]
      as well as a map
      \[ \on{Z}_{\der}^*(X) \to \K_0(X), \quad [Z/X] \mapsto [Z/X]^{\K} \]
      where $[Z/X]^K = [f_*(\sO_Z)]$ is the K-theoretic fundamental class (\ssecref{ssec:K fund}).
      There should be some natural quotient $\on{Z}_{\der}^*(X) \twoheadrightarrow \CH_{\der}^*(X)$ and an isomorphism
      \begin{equation}\label{eq:0agu0p71}
        \CH_{\der}^*(X)_\Q \to \Gr^* (\K_0(X))_\Q.
      \end{equation}

      The following is a sort of sanity check for this hypothesis.

      \begin{thm}
        For every qcqs algebraic space $X$ and every $[Z/X] \in \on{Z}_{\der}^n(X)$, we have
        \[ [Z/X]^{\K} \in \Gr^n (\K_0(X))_\Q. \]
        More generally, the map $f_* : \K_0(Z)\to \K_0(X)$ sends $\on{Fil}^k \K_0(Z)_\Q$ to $\on{Fil}^{k+n} \K_0(X)_\Q$ for every $k$.
      \end{thm}
      \begin{proof}
        By embedding $Z$ into a projective bundle over $X$ and using the projective bundle formula one reduces to the case of a quasi-smooth closed immersion.
        That case is proven in \cite[Thm.~2]{KhanGamma} using derived blow-ups \cite{KhanRydh} and the excess intersection formula \cite{KhanExcess} to further reduce to the case of virtual Cartier divisors.
      \end{proof}

      In particular, there is an induced Gysin map $f_*^\gamma : \Gr^* \K_0(Z)_\Q \to \Gr^* \K_0(X)_\Q$.
      One has moreover the following version of the Grothendieck--Riemann--Roch theorem, a direct generalization of the original formulation in \cite[Exp.~VIII, Thm.~3.6]{SGA6}:

      \begin{thm}
        Let $X$ be a qcqs algebraic space.
        For every quasi-smooth projective morphism $f : Z \to X$ of relative virtual dimension $n$, there is a commutative square
        \[ \begin{tikzcd}[matrix xscale=2]
          \K_0(Z) \ar{r}{f_*}\ar{d}{\ch}
          & \K_0(X) \ar{d}{\ch}
          \\
          \Gr^* \K_0(Z) _\Q \ar{r}{f_*^\gamma(- \cdot \on{Td}(\sL_{Z/X}^\vee))}
          & \Gr^* \K_0(X)_\Q
        \end{tikzcd} \]
        where $\on{Td}(\sL_{Z/X}^\vee)$ is the Todd class of the relative tangent complex.
      \end{thm}
      \begin{proof}
        Again, one reduces to the case of projective bundles, see \cite[Exp.~VIII, \S 5]{SGA6}, and the case of quasi-smooth closed immersions, see \cite[Thm.~3]{KhanGamma}.
      \end{proof}

      The theory of fundamental classes in motivic cohomology \cite{KhanVirtual} produces a \emph{ring} homomorphism
      \[ \on{Z}_{\der}^*(X) \to \H^{2*}(X, \bZ_{\mrm{mot}}(*)) \]
      which should descend to $\CH_{\der}^*(X)$.
      When $X$ is smooth the resulting map
      \begin{equation}\label{eq:iboqn}
        \CH_{\der}^*(X) \to \H^{2*}(X, \bZ_{\mrm{mot}}(*)) \simeq \CH^*(X)
      \end{equation}
      should be an isomorphism.
      Note that if the base field has resolution of singularities then every cycle in $\CH^n(X)$ can be represented by the fundamental class of a projective lci map, so there is at least some quotient of $\on{Z}_{\der}^*(X)$ for which \eqref{eq:iboqn} is an isomorphism.
      Whenever this is true, it means the intersection product on $\CH^*(X)$ can be represented geometrically by derived fibred products.

      Admittedly $\on{A}^*_{\mrm{der}}$ looks more like a cobordism theory than a cycle theory.
      In fact two completely different recent approaches to algebraic cobordism both yield constructions closely resembling $\on{A}^*_{\mrm{der}}(X)$.
      In \cite{EHKSY3} the authors considered a variant where the generators are required to be \emph{finite} instead of projective.
      That theory gives a model for cobordism of \emph{smooth} schemes and unfortunately requires a Nisnevich localization procedure which renders the end result fairly intractible.
      On the other hand Annala has recently constructed good theories of algebraic cobordism and Chow cohomology of singular schemes in characteristic zero \cite{AnnalaChar0}.
      After further simplification and partial extension to general bases in \cite{AnnalaYokura,AnnalaComp}, the generators are similar to those of $\on{A}^*_{\mrm{der}}$ and, according to a private communication from Annala, there is a simple set of relations one can impose on his ``precobordism'' such that \eqref{eq:0agu0p71} holds (under very mild hypotheses on $X$).
      Comparisons with the Levine--Weibel Chow groups \cite{LevineWeibel} have yet to be investigated.



\bibliographystyle{halphanum}

Institute of Mathematics\\
Academia Sinica\\
Taipei\\
10617 Taiwan

\end{document}